\documentclass[11pt,reqno]{amsproc}
\allowdisplaybreaks

\usepackage{amssymb}
\usepackage{stmaryrd}
\usepackage{MnSymbol}

\usepackage{color}
\usepackage{colortbl}
\usepackage{fullpage}

\usepackage{graphicx}
\usepackage{textcomp}
\usepackage{subfigure}
\usepackage{enumerate}
\usepackage{longtable} 
\usepackage[debug=false, colorlinks=true, pdfstartview=FitV,
linkcolor=blue, citecolor=blue, urlcolor=blue]{hyperref}
\usepackage[semicolon,square,authoryear,sort]{natbib}

\usepackage[usenames,dvipsnames,table]{xcolor}
\usepackage[most]{tcolorbox}

\usepackage{multirow} 
\usepackage{rotating} 
\usepackage{bigstrut} 
\usepackage{hhline} 

\usepackage[doipre={DOI:~}]{uri}

\newtheorem{theorem}{Theorem}[section]

\newtheorem{proposition}[theorem]{Proposition}

\newtheorem{remark}{Remark}[section]

\usepackage{morefloats}

\newlength{\drop}
\definecolor{amethyst}{rgb}{0.6, 0.4, 0.8}
\definecolor{burgundy}{rgb}{0.5, 0.0, 0.13}

\title{\textbf{Thermal regulation in thin vascular systems: \\
A sensitivity analysis}}

\author{\textbf{K.~B.~Nakshatrala} 
and \textbf{K.~Adhikari} \\
  {\small Department of Civil and Environmental Engineering \\ 
  University of Houston, Houston, Texas 77204, USA.} \\
  {\small
  \textbf{Author to whom correspondence should be addressed:} knakshatrala@uh.edu}
  }


\begin{document}

\begin{titlepage}
  \drop=0.1\textheight
  \centering
  \vspace*{\baselineskip}
  \rule{\textwidth}{1.6pt}\vspace*{-\baselineskip}\vspace*{2pt}
  \rule{\textwidth}{0.4pt}\\[\baselineskip]
       {\Large \textbf{\color{burgundy}
       Thermal regulation in thin vascular systems:\\[0.3\baselineskip]
       A sensitivity analysis}}\\[0.3\baselineskip]
       \rule{\textwidth}{0.4pt}\vspace*{-\baselineskip}\vspace{3.2pt}
       \rule{\textwidth}{1.6pt}\\[0.2\baselineskip]
       \scshape
       An e-print of the paper is available on arXiv. \par
       \vspace*{0.2\baselineskip}
       Authored by \\[0.2\baselineskip]

  {\Large K.~B.~Nakshatrala\par}
  {\itshape Associate Professor, Department of Civil \& Environmental Engineering \\
  University of Houston, Houston, Texas 77204. \\
  \textbf{phone:} +1-713-743-4418, \textbf{e-mail:} knakshatrala@uh.edu \\
  \textbf{website:} http://www.cive.uh.edu/faculty/nakshatrala}\\[0.2\baselineskip]
  
  {\Large K.~Adhikari\par}
  {\itshape Graduate Student, Department of Civil \& Environmental Engineering \\
  University of Houston, Houston, Texas 77204.} \\[0.1\baselineskip]
  
  \begin{figure*}[ht]
    \centering
    \includegraphics[scale=0.9]{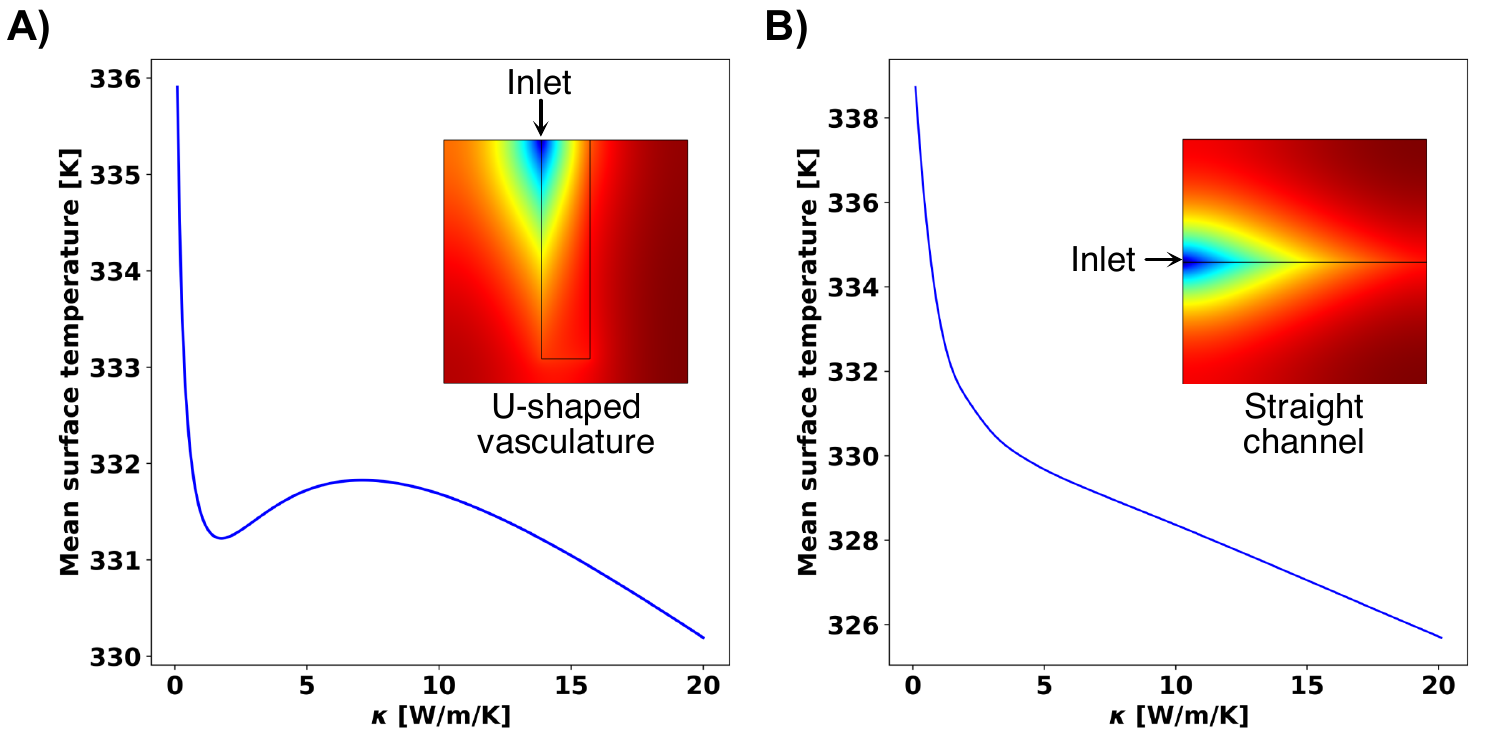}
    
  \emph{This graphical abstract highlights the chief finding of the sensitivity analysis undertaken in this paper, addressing active-cooling via fluid flow through an embedded microvasculature. The mean surface temperature (MST)---a popular performance metric---monotonically decreases with an increase in the fluid's heat capacity rate. However, MST variation can be non-monotonic to the change in the host material's thermal conductivity---a surprising result. A) When countercurrent heat exchange is significant (e.g., in the case of U-shaped vasculature), the sensitivity of MST to conductivity can be positive or negative. B) The remarked sensitivity decreases monotonically when countercurrent heat exchange is absent or insignificant (e.g., straight channel).}
  \end{figure*}
  
  \vfill
  {\scshape 2023} \\
  {\small Computational \& Applied Mechanics Laboratory} \par
\end{titlepage}

\begin{abstract}
One of the ways natural and synthetic systems regulate temperature is via circulating fluids through vasculatures embedded within their bodies. Because of the flexibility and availability of proven fabrication techniques, vascular-based thermal regulation is attractive for thin microvascular systems. Although preliminary designs and experiments demonstrate the feasibility of thermal modulation by pushing fluid through embedded micro-vasculatures, one has yet to optimize the performance before translating the concept into real-world applications. It will be beneficial to know how two vital design variables---host material’s thermal conductivity and fluid’s heat capacity rate---affect a thermal regulation system’s performance, quantified in terms of the mean surface temperature. This paper fills the remarked inadequacy by performing adjoint-based sensitivity analysis and unravels a surprising non-monotonic trend. Increasing thermal conductivity can either increase or decrease the mean surface temperature; the increase happens if countercurrent heat exchange---transfer of heat from one segment of the vasculature to another---is significant.
In contrast, increasing the heat capacity rate will invariably lower the mean surface temperature, for which we provide mathematical proof. The reported results (a) dispose of some misunderstandings in the literature, especially on the effect of the host material’s thermal conductivity, (b) reveal the role of countercurrent heat exchange in altering the effects of design variables, and (c) guide designers to realize efficient microvascular active-cooling systems. The analysis and findings will advance the field of thermal regulation both on theoretical and practical fronts.
\end{abstract}

\maketitle

\vspace{-0.3in}


\setcounter{figure}{0}


\section*{A LIST OF MATHEMATICAL SYMBOLS AND ABBREVIATIONS}

\begin{longtable}{ll}\hline
  \multicolumn{1}{c}{\textbf{Symbol}} & \multicolumn{1}{c}{\textbf{Definition}} \\
  \hline \multicolumn{2}{c}{\emph{Operators}} \\ \hline
  $\llbracket \cdot \rrbracket$ & jump operator across the vasculature $\Sigma$ \\
  $\llangle \cdot \rrangle$ & average operator across the vasculature $\Sigma$ \\
  $\mathrm{div}[\cdot]$ & spatial divergence operator \\
  $\mathrm{grad}[\cdot]$ & spatial gradient operator \\
  \hline \multicolumn{2}{c}{\emph{Geometry-related quantities}} \\ \hline
  $\Omega$ & domain (i.e., mid-surface of the body) \\
  $\partial \Omega$ & boundary of the domain  \\
  $\Gamma^{\vartheta}$ & part of the boundary with prescribed temperature \\
  $\Gamma^{q}$ & part of the boundary with prescribed heat flux \\
  $\Sigma$ & curve representing the vasculature \\
  $d$ & thickness of the body \\ 
  $\widehat{\mathbf{n}}(\mathbf{x})$ & unit outward normal vector to the boundary \\
  $\widehat{\mathbf{n}}^{\pm}(\mathbf{x})$ & unit outward normal vector on either sides of $\Sigma$ \\
  $s$ & normalized arc-length along $\Sigma$, measured from the inlet under forward flow \\ 
  $\widehat{\mathbf{t}}(\mathbf{x})$ & unit tangential vector along $\Sigma$ \\
  $\mathbf{x}$ & a spatial point \\
  \hline \multicolumn{2}{c}{\emph{Solution fields}} \\ \hline
  $\vartheta(\mathbf{x})$ & temperature (scalar) field \\
  $\vartheta^{(f)}(\mathbf{x})$ & temperature field under forward flow conditions \\
  $\vartheta^{(r)}(\mathbf{x})$ & temperature field under reverse flow conditions \\
  $\vartheta_{\mathrm{outlet}}$ & outlet temperature \\
  $\mathbf{q}(\mathbf{x})$ & heat flux vector field \\
  \hline \multicolumn{2}{c}{\emph{Prescribed quantities}} \\ \hline
  $\vartheta_{\mathrm{amb}}$ & ambient temperature \\ 
  $\vartheta_{\mathrm{inlet}}$ & inlet temperature \\
  $f(\mathbf{x})$ & power supplied by heat source \\
  $f_0$ & a constant power by heat source \\ 
  $Q$ & volumetric flow rate in the vasculature \\
  \hline \multicolumn{2}{c}{\emph{Material and surface properties}} \\ \hline
  $\kappa $ & host material's thermal conductivity \\
  $\rho_f$ & fluid's density \\
  $c_f$ & fluid's specific heat capacity \\
  $h_{T}$ & (combined) heat transfer coefficient \\
  \hline \multicolumn{2}{c}{\emph{Other symbols}} \\ \hline
  $\eta$ & thermal efficiency \\
  $\overline{\vartheta}$ & spatial mean of the temperature field (i.e., mean surface temperature) \\
  $\overline{\vartheta}_{\mathrm{HSS}}$ & spatial mean of hot steady-state temperature \\
  $\dot{m}$ & mass flow rate in the vasculature, $\dot{m} = \rho_f \, Q$ \\
  $\chi$ & heat capacity rate of the fluid, $\chi = \dot{m} \, c_f$ \\
  $\xi(\mathbf{x})$ & an arbitrary design field variable \\ 
  $\Phi[\cdot]$ & objective functional \\
  $D\Phi[\xi(\mathbf{x})]$ & Fr\'echet derivative of functional $\Phi$ with respect to $\xi(\mathbf{x})$ \\ 
  \hline \multicolumn{2}{c}{\emph{Abbreviations}} \\ \hline
  \textsf{HSS} & hot steady-state \\ 
  \textsf{MST} & mean surface temperature \\
  \textsf{QoI} & quantity of interest \\
  \hline
\end{longtable}

\setcounter{table}{0}

\section{INTRODUCTION AND MOTIVATION}
\label{Sec:S1_Sensitivity_Intro}

Moving fluids through embedded vasculatures offers environmental-friendly solutions to many thermal regulation applications. For example, a geothermal system---a renewable energy source---comprises a network of pipes in the ground and thrusts a ct ofluid (typically water) to extract thermal energy from the subsurface to heat homes and appliances \citep{barbier2002geothermal}. Another application, which is becoming popular, is using vascular-based thermal modulation to de-ice grounded aircraft instead of toxic anti-freeze chemicals, which pollute groundwater and soils if not correctly handled \citep{murray1997toxicological}. In separate developments, researchers avail fluid flow in embedded vesicles for controlling temperature fields to achieve other vital functionalities in synthetic composites, such as electromagnetic modulation \citep{devi2021microvascular}, \emph{in situ} self-healing \citep{Snyder_self_healing_2022}, and reduced thermal stresses \citep{ccetkin2015vascularization}. Given these and many other potential applications, recent progress across several scientific fields provides perfect opportunities to spur the growth of thermal regulation in \emph{microvascular} composite and metal systems. These fields include \emph{experimental heat transfer} (e.g., optical imaging \citep{mayinger2013optical}), \emph{fabrication techniques} (e.g., 3D printing \citep{nguyen2018recent}, frontal polymerization \citep{robertson2018rapid}), \emph{modeling methods} (e.g., reduced-order models \citep{tan20173d,nakshatrala2022modeling}), \emph{numerical formulations} (e.g., stabilized formulations for convection-dominated problems \citep{hughes1995multiscale,masud2004multiscale,franca1992stabilized,franca2006revisiting,codina2000stabilized,hsu2010improving,turner2011stabilized}), and \emph{design approaches} (e.g., topology optimization \citep{alexandersen2014topology,dbouk2017review,ahmed2018optimization,pejman2019gradient}). Nevertheless, two aspects need a thorough examination in perfecting such microvascular systems. 

\emph{First}, one must identify suitable quantities of interest (QoIs) that can assess the system's performance; for example, thermal efficiency is a popular performance metric in heat transfer and thermodynamics studies. In mathematical optimization jargon, an objective function is a popular alternative to QoI; herein, we use these two terms interchangeably. Upon a selection, a designer's goal would be to ``extremize" (maximize or minimize, depending on the choice) the chosen objective function. Selecting appropriate objective functions for thermal regulation---from either a design perspective or computational appeal---is still an unsettled question, certainly warranting an in-depth study; nonetheless, such an investigation is beyond this paper's scope. That said, studies involving thin members often aim to minimize the mean surface temperature (MST) \citep{pejman2019gradient}. Also, as this paper shows, minimizing MST is equivalent to maximizing cooling efficiency. Because of these two reasons, we take MST as the quantity of interest and defer exploring alternatives to a follow-up article. But a natural question arises: 
\begin{enumerate}[(Q1)]
    \item What are the ramifications of \emph{minimizing the mean surface temperature} on other thermal characteristics? Said differently, what equivalent changes does it bring to the system?
\end{enumerate}

\emph{Second}, one needs to identify an appropriate set of design variables: geometrical, material, and input parameters that a designer can vary to alter the system’s performance. Several studies have explored geometrical attributes, such as altering vasculature layout (e.g., spiral) and changing the spacing among branches \citep{bejan1995thermal,aragon2008design,aragon2011multi,tan2015nurbs,mcelroy2015optimisation,pejman2019gradient}. However, prior studies have not adequately investigated the effects of the host material’s thermal conductivity and the fluid’s heat capacity rate (product of fluid's heat capacity, a material parameter, and volumetric flow rate, an input parameter). Further, a diligent look at the literature reveals an unresolved issue: studies have indicated that the host material’s conductivity does not significantly affect the mean surface temperature or thermal efficiency (e.g., see \citep{devi2021microvascular}). But a simple thought experiment conjectures the possibility of the opposite: A higher conductive host material will offer lower resistance to the heat flowing towards the sink---the coolant flowing in the vasculature. Thus, the flowing fluid will take out a higher percentage of supplied heat to the system---meaning a higher thermal efficiency and a lower MST. Ergo, there is a clear-cut dichotomy. Indeed, as we will show later in this paper, the situation is more intriguing than the discussion heretofore. Thus, a related key question is:
\begin{enumerate}[(Q2)]
\item How does altering the design variables---the fluid's heat capacity rate and the host material's thermal capacity---vary the mean surface temperature? 
\end{enumerate}

This paper comprehensively addresses the above-outlined two questions: We will use integral theorems to answer the first question, while a mathematical sensitivity analysis alongside finite element simulations will address the second. The workhorse for both integral theorems and sensitivity analysis is a reduced-order mathematical model for which \citet{nakshatrala2022modeling} has recently provided a theoretical underpinning. In addition, we use the adjoint state method for the sensitivity analysis to assess how the design variables affect the quantity of interest. The adjoint state method---a powerful technique to calculate design sensitivities---is widely used in optimal control theory \citep{lions1971optimal}, inverse problems \citep{givoli2021tutorial}, PDE-constrained optimization \citep{bradley2013pde}, tomography problems \citep{tromp2005seismic}, shape and topology optimization \citep{jameson2003aerodynamic,bendsoe_sigmund_2013topology}, material design \citep{nakshatrala2013nonlinear,nakshatrala2022pose}, and dynamic check-pointing \citep{wang2009minimal}, to name a few. The chief advantage of the adjoint state method is that it circumvents an explicit calculation of the solution field's sensitivity to a design variable. This circumvention is attractive to this study, as our primary intention is to assess the sign of the sensitivity (positive or negative): whether the design variable promotes or hinders the QoI. The adjoint state method allows us to do so without explicitly calculating the solution field. 

The results presented in this paper provide a deeper understanding of active cooling and pave a systematic path for a mathematical-driven material design of thermal regulation systems. The plan for the rest of this article is as follows. We first present the governing equations for the direct problem: a reduced-order mathematical model describing vascular-based thermal regulation (\S\ref{Sec:S2_Sensitivity_Forward}). Using this model, we then deduce the ramifications of minimizing the mean surface temperature---a popular quantity of interest---on thermal efficiency and other thermal attributes of the system (\S\ref{Sec:S3_Sensitivity_Ramifications}). Next, using the adjoint state method, we calculate the sensitivity of the mean surface temperature to the fluid's heat capacity rate (\S\ref{Sec:S4_Sensitivity_Heat_capacity_rate}). Following a similar procedure, we estimate the sensitivity of the same quantity of interest to the thermal conductivity of the host material (\S\ref{Sec:S5_Sensitivity_Conductivity}). After that, we verify the newfound theoretical results using numerical simulations (\S\ref{Sec:S6_Sensitivity_NR}). Finally, we draw concluding remarks and put forth potential future works (\S\ref{Sec:S7_Sensitivity_Closure}).

\section{MATHEMATICAL DESCRIPTION OF THE DIRECT PROBLEM}
\label{Sec:S2_Sensitivity_Forward}
\textbf{Figure~\ref{Fig:Sensitivity_Pictorial_Setup}} depicts a typical active-cooling setup: Consider a thin body whose thickness is much smaller than its other characteristic dimensions. The body contains a connected vasculature, with inlet and outlet openings on the lateral surface, which is otherwise insulated. A heat source supplies power to one of the transverse faces while the other is free to convect and radiate. A liquid flows through the vasculature, enabling active-cooling. We assume radiation is relatively minor, and, as commonly done, we lump its contribution to the convective component using the combined/overall heat transfer coefficient \citep{kaminski2017introduction}. The lumping makes the resulting mathematical model linear, making the sensitivity analysis analytically tractable. Since the body is thin, a full three-dimensional model is inordinate. So, this paper avails a reduced-order model defined in a two-dimensional domain---the mid-surface of the slender body. Accordingly, we model the vasculature as a curve within this domain rather than resolving its cross-sectional area. 

For mathematical description, $\Omega \subset \mathbb{R}^{2}$ denotes the domain, $\partial \Omega$ the boundary, and $d$ the body's thickness. For technical purposes, we assume $\Omega$ to be open and bounded and $\partial \Omega$ piece-wise smooth. A spatial point is denoted by $\mathbf{x} \in \Omega \cup \partial \Omega$ and the outward unit normal vector to the boundary by $\widehat{\mathbf{n}}(\mathbf{x})$. The spatial divergence and gradient operators are denoted by $\mathrm{div}[\cdot]$ and $\mathrm{grad}[\cdot]$, respectively. We denote the temperature field in the domain by $\vartheta(\mathbf{x})$ and the heat flux vector field by $\mathbf{q}(\mathbf{x})$. $\vartheta_{\mathrm{amb}}$ denotes the ambient temperature---the surrounding environment's temperature. The boundary is divided into two complementary parts: $\Gamma^{\vartheta} \cup \Gamma^{q} = \partial \Omega$. $\Gamma^{\vartheta}$ is the part of the boundary on which temperature (i.e., Dirichlet boundary condition) is prescribed, while $\Gamma^{q}$ is that part of the boundary on which heat flux (i.e., Neumann boundary condition) is prescribed. For mathematical well-posedness, we require $\Gamma^{\vartheta} \cap \Gamma^{q} = \emptyset$.

\begin{figure}[ht]
    \centering
    \includegraphics[scale=0.8]{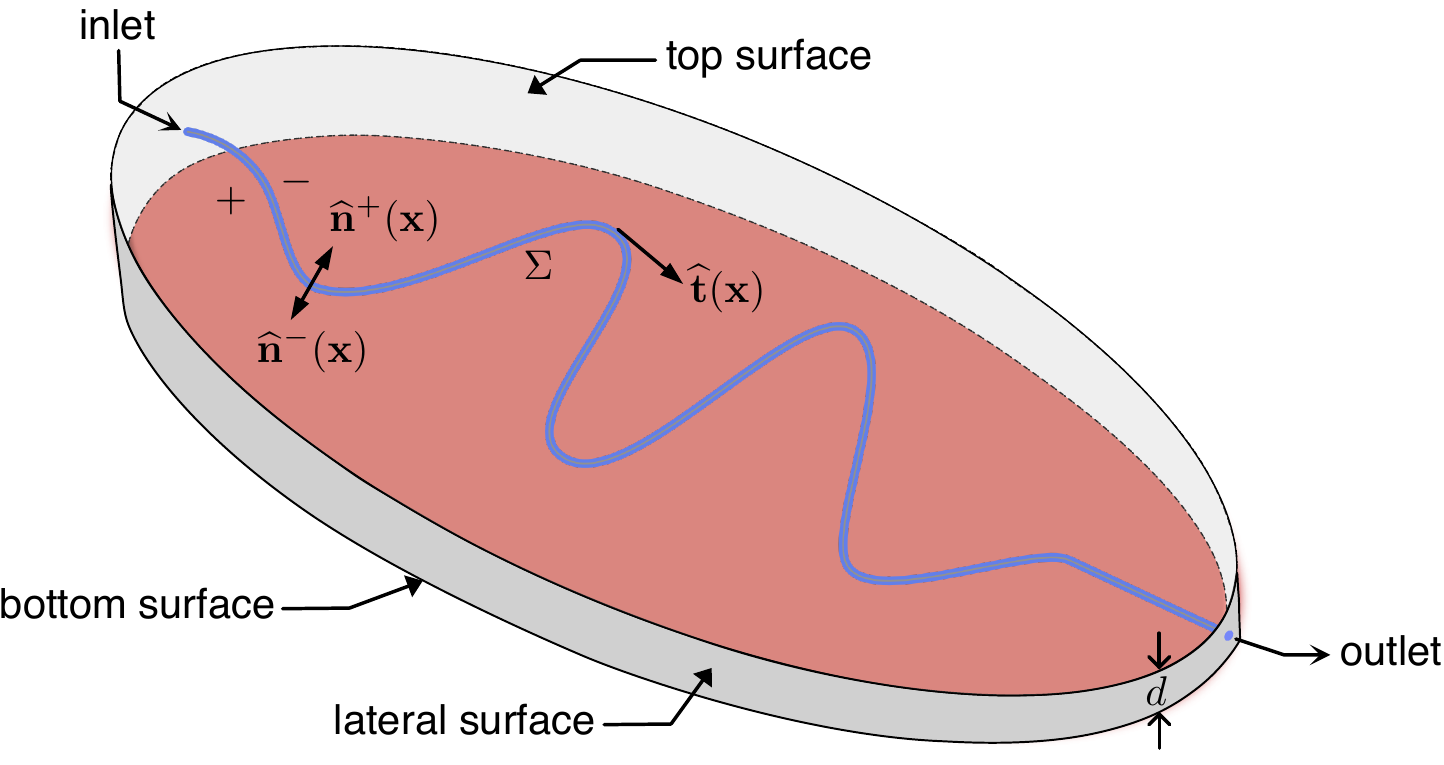}
    \caption{This figure shows a typical thin vascular-based active-cooling setup. The domain with thickness $d$ contains a vasculature $\Sigma$ through which a fluid (coolant) flows. A heat source is at the bottom transverse surface while the top is free to convect and radiate (i.e., exposed to the environment). The lateral surface is adiabatic.  $\widehat{\mathbf{n}}^{\pm}(\mathbf{x})$ denote the unit normals on either side of $\Sigma$. The unit tangent vector along the vasculature---starting from the inlet to outlet---is denoted by $\widehat{\mathbf{t}}(\mathbf{x})$.}
    \label{Fig:Sensitivity_Pictorial_Setup}
\end{figure}

Fluid passage through the vasculature enables active-cooling: heat transfers between the host solid and flowing fluid. $\rho_f$ and $c_f$ denote the fluid's density and specific heat capacity, respectively. $Q$ represents the volumetric flow rate within the vasculature. Thus, the mass flow rate reads: 
\begin{align}
    \label{Eqn:Sensitivity_mass_flow_rate}
    \dot{m} = \rho_{f} \, Q 
\end{align}
The heat capacity rate---one of the design variables considered in this paper---is then defined as:
\begin{align}
     \label{Eqn:Sensitivity_heat_capacity_rate}
    \chi = \dot{m} \, c_f 
\end{align}
Note that $\chi$ is a product of fluid properties ($\rho_f$ and $c_f$) and an input parameter ($Q$). 

The curve representing the vasculature is denoted by $\Sigma$ and is parameterized using the normalized arc-length $s$, with $s = 0$ at the inlet and $s = 1$ at the outlet. The unit tangent vector in the direction of increasing arc-length at a spatial point on this curve is denoted by $\widehat{\mathbf{t}}(\mathbf{x})$. We denote the inlet and outlet temperatures by $\vartheta_{\mathrm{inlet}}$ and $\vartheta_{\mathrm{outlet}}$, respectively. $\vartheta_{\mathrm{inlet}}$ is a prescribed input, while $\vartheta_{\mathrm{outlet}}$ is a part of the solution---generally unknown without solving the boundary value problem. Throughout this paper, we assume $\vartheta_{\mathrm{inlet}} = \vartheta_{\mathrm{amb}}$. 

\subsection{Average and jump operators}
The fluid flow within the vasculature creates a jump in the heat flux across the curve $\Sigma$. On this account, we introduce the necessary notation to describe the remarked jumps mathematically. We denote the outward unit normals on either side of $\Sigma$ by $\widehat{\mathbf{n}}^{+}(\mathbf{x})$ and $\widehat{\mathbf{n}}^{-}(\mathbf{x})$ (see Fig.~\ref{Fig:Sensitivity_Pictorial_Setup}). Assigning the signs to these normals can be arbitrary. For instance, label one of the outward normals as positive; the other---pointing in the opposite direction---will be negative. These two normals satisfy:
\begin{align}
    \widehat{\mathbf{n}}^{+}(\mathbf{x}) +     \widehat{\mathbf{n}}^{-}(\mathbf{x}) = \mathbf{0}  
    \quad \mathrm{and} \quad
    \widehat{\mathbf{n}}^{\pm}(\mathbf{x}) \bullet \widehat{\mathbf{t}}(\mathbf{x}) = 0 
    \qquad \forall \mathbf{x} \in \Sigma
\end{align}
where $\bullet$ denotes the (Euclidean) dot product. 

Given a field $\gamma(\mathbf{x})$, we denote the \emph{limiting values} on either side of $\Sigma$ by $\gamma^{+}(\mathbf{x})$ and $\gamma^{-}(\mathbf{x})$. Mathematically, 
\begin{align}
    \label{Eqn:Sensitivity_limiting_values}
    \gamma^{+}(\mathbf{x}) := \lim_{\epsilon \rightarrow 0} \; \gamma(\mathbf{x} - |\epsilon| \, \mathbf{n}^{+}(\mathbf{x})) 
    \quad \mathrm{and} \quad 
    \gamma^{-}(\mathbf{x}) := \lim_{\epsilon \rightarrow 0} \; \gamma(\mathbf{x} - |\epsilon| \, \mathbf{n}^{-}(\mathbf{x})) 
\end{align}
We then define the \emph{average operator} across $\Sigma$ for a scalar field $\alpha(\mathbf{x})$ and a vector field $\mathbf{a}(\mathbf{x})$ as follows: 
\begin{subequations}
    \label{Eqn:Sensitivity_average_operator} 
    \begin{align}
        \label{Eqn:Sensitivity_avg_operator_scalar} 
        \llangle \alpha(\mathbf{x}) \rrangle 
        &= \frac{1}{2} \Big(\alpha^{+}(\mathbf{x}) + \alpha^{-}(\mathbf{x}) \Big) \\
        \label{Eqn:Sensitivity_avg_operator_vector} 
        \llangle \mathbf{a}(\mathbf{x}) \rrangle 
        &= \frac{1}{2} \Big(\mathbf{a}^{+}(\mathbf{x}) + \mathbf{a}^{-}(\mathbf{x}) \Big) 
    \end{align}
\end{subequations}
For these two fields, the \emph{jump operator} across $\Sigma$ is defined as follows: 
\begin{subequations} 
    \label{Eqn:Sensitivity_jump_operator} 
    \begin{align}
        \label{Eqn:Sensitivity_jump_operator_scalar} 
        &\llbracket \alpha(\mathbf{x}) \rrbracket 
        = \alpha^{+}(\mathbf{x}) \, 
        \widehat{\mathbf{n}}^{+}(\mathbf{x}) 
        + \alpha^{-}(\mathbf{x}) \, 
        \widehat{\mathbf{n}}^{-}(\mathbf{x}) \\
        \label{Eqn:Sensitivity_jump_operator_vector} 
        &\llbracket\mathbf{a}(\mathbf{x}) \rrbracket = \mathbf{a}^{+}(\mathbf{x})
        \bullet \widehat{\mathbf{n}}^{+}(\mathbf{x}) 
        + \mathbf{a}^{-}(\mathbf{x}) 
        \bullet \widehat{\mathbf{n}}^{-}(\mathbf{x}) 
    \end{align}
\end{subequations}
Note that the jump operator acts on a scalar field to produce a vector field and \emph{vice versa}.\footnote{There is an alternative definition used in the literature for the jump operator: $\llbracket a(\mathbf{x}) \rrbracket = a^{+}(\mathbf{x}) - a^{-}(\mathbf{x})$. For example, see \citep{chadwick2012continuum}. The definition used in this paper (i.e., Eq.~\eqref{Eqn:Sensitivity_jump_operator}) is symmetric with respect to $+$ and $-$ sub-domains and is more convenient for writing Green's theorem and executing other calculations.} For scalar fields $\alpha(\mathbf{x})$ and $\beta(\mathbf{x})$ and vector fields $\mathbf{a}(\mathbf{x})$ and $\mathbf{b}(\mathbf{x})$, the following identities (see Appendix for a derivation):
\begin{subequations}
    \label{Eqn:Sensitivities_avg_jump_identities}
    \begin{align}
         \label{Eqn:Sensitivities_avg_jump_identities_a}
        &\llbracket \alpha(\mathbf{x}) \,  \mathbf{a}(\mathbf{x}) \rrbracket 
        = \llbracket \alpha(\mathbf{x}) \rrbracket \bullet
        \llangle \mathbf{a}(\mathbf{x}) \rrangle
        + \llangle \alpha(\mathbf{x}) \rrangle \, 
        \llbracket \mathbf{a}(\mathbf{x}) \rrbracket \\
         \label{Eqn:Sensitivities_avg_jump_identities_b}
        &\llbracket \mathbf{a}(\mathbf{x}) \bullet  \mathbf{b}(\mathbf{x}) \rrbracket 
        = \llbracket \mathbf{a}(\mathbf{x}) \rrbracket 
        \llangle \mathbf{b}(\mathbf{x}) \rrangle
        + \llangle \mathbf{a}(\mathbf{x}) \rrangle \, 
        \llbracket \mathbf{b}(\mathbf{x}) \rrbracket 
    \end{align}
\end{subequations}
These identities will be used later in sensitivity analysis to derive the adjoint state problem. 

\subsection{Reduced-order mathematical model}
The model considers three modes of heat transfer. Newton's law of cooling accounts for the convection on a free surface. A jump condition---an energy balance across the vasculature---models the heat transfer between the circulating fluid and the host solid. The Fourier model describes the conduction within the host solid: 
\begin{align}
    \label{Eqn:Sensitivity_Fourier_model}
    \mathbf{q}(\mathbf{x}) = - \kappa(\mathbf{x}) \,  
    \mathrm{grad}[\vartheta(\mathbf{x})]
\end{align}
where $\kappa(\mathbf{x})$ is the coefficient of thermal conductivity.

The governing equations for the reduced-order model, describing thermal regulation, are 
\begin{subequations}
\begin{alignat}{2}
    \label{Eqn:Sensitivity_BoE} 
    -&d \, \mathrm{div}[\kappa(\mathbf{x}) \, 
    \mathrm{grad}[\vartheta(\mathbf{x})]] 
    = f(\mathbf{x})
    - h_{T} \, (\vartheta(\mathbf{x}) - \vartheta_{\mathrm{amb}}) 
    && \quad \mathrm{in} \; \Omega \setminus \Sigma \\
    \label{Eqn:Sensitivity_q_jump_condition} 
    -&d \, \llbracket\kappa(\mathbf{x}) \,
    \mathrm{grad}[\vartheta(\mathbf{x})]\rrbracket 
    = \chi \,  \llangle \mathrm{grad}[\vartheta(\mathbf{x})] 
    \bullet \widehat{\mathbf{t}}(\mathbf{x}) \rrangle
    && \quad \mathrm{on} \; \Sigma \\
    \label{Eqn:Sensitivity_temp_jump_condition} 
    &\llbracket\vartheta(\mathbf{x})\rrbracket = 0  
    && \quad \mathrm{on} \; \Sigma \\
    \label{Eqn:Sensitivity_q_BC} 
    -&d \, \widehat{\mathbf{n}}(\mathbf{x}) \bullet \kappa(\mathbf{x}) \mathrm{grad}[\vartheta(\mathbf{x})] 
    = 0 
    && \quad \mathrm{on} \; \partial \Omega \\
    \label{Eqn:Sensitivity_inlet_BC} 
    &\llangle \vartheta(\mathbf{x})\rrangle = \vartheta_{\mathrm{inlet}} 
    = \vartheta_{\mathrm{amb}} 
    && \quad \mathrm{at} \; s = 0 \; \mathrm{on} \; \Sigma 
\end{alignat}
\end{subequations}
where $h_T$ is the combined heat transfer coefficient, and $f(\mathbf{x})$ is the power supplied by the heat source. For underlying assumptions and a thorough mathematical analysis of this reduced-order model, see \citep{nakshatrala2022modeling}. Reduced-order models, similar to the one presented above, have been employed in prior active-cooling studies: for example, to establish invariants under flow reversal \citep{nakshatrala2022invariance}, to get optimized vascular layouts using topology or shape optimization---or both \citep{pejman2019gradient,najafi2015gradient,tan20173d}, to develop analysis numerical frameworks \citep{tan2015nurbs}, to name a few. However, these studies' focus has been different and not towards addressing the two questions central to this paper, outlined in the introduction. In the rest of the article, we will use mathematical analysis and sensitivity analysis on the reduced-order model to answer these questions. 

In the parlance of sensitivity analysis and inverse problems, the above boundary value problem is often referred to as the direct problem.\footnote{In the literature, another popular name for the direct problem is the forward problem. As we will show later, the sensitivity analysis (under the adjoint state method) avails boundary value problems under \emph{forward} and reverse flow conditions (i.e., swapping the inlet and outlet). To avoid a potential confusion, we do not adopt the forward problem terminology---but use the direct problem instead.} Two remarks are warranted to relate the form in which the governing equations are presented above to those used by other works. 

\begin{remark}
\label{Remark:Sensitivity_Tangential_component}
In Eq.~\eqref{Eqn:Sensitivity_q_jump_condition}, $\mathrm{grad}[\vartheta]\bullet \widehat{\mathbf{t}}(\mathbf{x})$ represents the tangential derivative of the temperature field along the vasculature. Equation \eqref{Eqn:Sensitivity_temp_jump_condition} implies that the temperature is continuous across $\Sigma$. Since the unit tangent vector $\widehat{\mathbf{t}}(\mathbf{x})$ remains the same on either side of the vasculature, the tangential derivative $\mathrm{grad}[\vartheta]\bullet \widehat{\mathbf{t}}(\mathbf{x})$ will also be continuous across $\Sigma$. So, the average operator can be dropped on the right side of Eq.~\eqref{Eqn:Sensitivity_q_jump_condition}; one can alternatively write it as 
\begin{align}
    -d \, \llbracket\kappa(\mathbf{x}) \mathrm{grad}[\vartheta]\rrbracket 
    = \chi \,  \mathrm{grad}[\vartheta] \bullet \widehat{\mathbf{t}}(\mathbf{x}) 
    \quad \mathrm{on} \; \Sigma 
\end{align}
\cite{nakshatrala2022modeling} has utilized the above alternative equation in their mathematical analysis to account for the heat transfer across the vasculature. But, for carrying out the sensitivity analysis using the adjoint method, Eq.~\eqref{Eqn:Sensitivity_q_jump_condition} is better suited than the above alternative. 
\end{remark}

\begin{remark}
\label{Remark:Sensitivity_Inlet_condition}
Equation \eqref{Eqn:Sensitivity_inlet_BC} assigns the prescribed inlet temperature to $\llangle \vartheta(\mathbf{x}) \rrangle$ at $s = 0$ on $\Sigma$; that is, the assignment is to the average of the temperatures on the either sides of the  vasculature at the inlet. Because of the continuity of the temperature field across the vasculature (i.e., Eq.~\eqref{Eqn:Sensitivity_temp_jump_condition}), the average operator can be dropped in writing the initial condition: 
\begin{align}
    \vartheta(\mathbf{x}) 
    = \vartheta_{\mathrm{inlet}} = \vartheta_{\mathrm{amb}}  
    \quad \mathrm{at} \; s = 0 \; \mathrm{on} \; \Sigma 
\end{align}
Some works prefer the above equation, for example, see \citep{nakshatrala2022modeling}. However, for our paper, Eq.~\eqref{Eqn:Sensitivity_inlet_BC} is a better (equivalent) alternative, as it will simplify the derivations in the sensitivity analysis. 

Also, since $\vartheta_{\mathrm{inlet}}$ and $\vartheta_{\mathrm{amb}}$ are constants and do not experience jumps across the vasculature, we rewrite Eq.~\eqref{Eqn:Sensitivity_inlet_BC}, in derivations in this paper, as: 
\begin{align}
    \llangle \vartheta(\mathbf{x}) 
    - \vartheta_{\mathrm{inlet}} 
    \rrangle 
    = \llangle \vartheta(\mathbf{x}) 
    - \vartheta_{\mathrm{amb}} 
    \rrangle 
    = 0  
    \quad \mathrm{at} \; s = 0 \; \mathrm{on} \; \Sigma 
\end{align}
\end{remark}

\subsection{Useful definitions}
The \emph{mean surface temperature} is defined as 
\begin{align}
    \label{Eqn:ROM_Sensitivity_MST}
    \overline{\vartheta} 
    := \frac{1}{\mathrm{meas}(\Omega)}\int_{\Omega} \vartheta(\mathbf{x}) \, \mathrm{d} \Omega 
\end{align}
where $\mathrm{meas}(\Omega)$ denotes the (set) measure of $\Omega$. Since $\Omega$ is a surface, $\mathrm{meas}(\Omega)$ is the area of $\Omega$. The \emph{thermal efficiency}---referred to as the \emph{cooling efficiency} in the context of active-cooling---is defined as the ratio of the rate at which heat is carried away by the flowing fluid (within the vasculature) to the total power supplied by the heater. Mathematically, 
\begin{align}
    \label{Eqn:Sensitivity_eta_definition_general}
    \eta 
    := \left( \int_{\Omega} f(\mathbf{x}) \, \mathrm{d} \Omega \right)^{-1} \chi \, (\vartheta_{\mathrm{outlet}} - \vartheta_{\mathrm{inlet}})
\end{align}
If the applied heat supply is uniform over the entire domain (i.e., $f(\mathbf{x}) = f_0$), then we have:
\begin{align}
    \label{Eqn:Sensitivity_eta_definition_f0}
    \eta = \frac{\chi \, (\vartheta_{\mathrm{outlet}} - \vartheta_{\mathrm{inlet}})}{f_0 \, \mathrm{meas}(\Omega)}
\end{align}

We refer to the steady-state achieved without active-cooling as the \emph{hot steady-state} (HSS) and denote the corresponding temperature field as $\vartheta_{\mathrm{HSS}}(\mathbf{x})$. That is, HSS occurs when $\chi = 0$ and the temperature at the inlet is not prescribed. With this definition and by availing Eqs.~\eqref{Eqn:Sensitivity_BoE}--\eqref{Eqn:Sensitivity_inlet_BC}, the \emph{mean hot steady-state temperature} can be written as 
\begin{align}
    \label{Eqn:Sensitivity_Theta_HSS_definition}
    \overline{\vartheta}_{\mathrm{HSS}} 
    := \frac{1}{\mathrm{meas}(\Omega)} \int_{\Omega} \vartheta_{\mathrm{HSS}}(\mathbf{x}) \, \mathrm{d} \Omega 
    = \vartheta_{\mathrm{amb}} + \frac{1}{h_T \, \mathrm{meas}(\Omega)} \int_{\Omega} f(\mathbf{x}) \, \mathrm{d} \Omega 
\end{align}
If $f(\mathbf{x}) = f_0$ (i.e., uniform power supplied by the heater), we have: 
\begin{align}
    \overline{\vartheta}_{\mathrm{HSS}} 
    = \vartheta_{\mathrm{amb}} + \frac{f_0}{h_T}
\end{align}
As revealed by the above equation, for the chosen thermal regulation setup (see Fig.~\ref{Fig:Sensitivity_Pictorial_Setup}),  $\overline{\vartheta}_{\mathrm{HSS}}$ is independent of the host material's conductivity.

\subsection{Forward and reverse flows}
\label{Subsec:Sensitivity_forward_reverse_flows}
Shown later, the sensitivity analysis avails a boundary value problem under reverse flow conditions: that is, the inlet and outlet locations are swapped, and hence, the fluid flows in the opposite direction within the vasculature. The solution under forward flow conditions will be denoted by $\vartheta^{(f)}(\mathbf{x})$, which satisfies Eqs.~\eqref{Eqn:Sensitivity_BoE}--\eqref{Eqn:Sensitivity_inlet_BC}. While $\vartheta^{(r)}(\mathbf{x})$ denotes the solution under reverse flow conditions, satisfying the following boundary value problem:  
\begin{subequations}
    \begin{alignat}{2}
        \label{Eqn:Sensitivity_BoE_Reverse} 
        -&d \, \mathrm{div}[\kappa(\mathbf{x}) \, 
        \mathrm{grad}[\vartheta^{(r)}(\mathbf{x})]] 
        = f(\mathbf{x}) 
        - h_{T} \, (\vartheta^{(r)}(\mathbf{x}) - \vartheta_{\mathrm{amb}}) 
        && \quad \mathrm{in} \; \Omega \setminus \Sigma \\
        \label{Eqn:Sensitivity_q_jump_condition_Reverse} 
        -&d \, \llbracket\kappa(\mathbf{x}) \, \mathrm{grad}[\vartheta^{(r)}(\mathbf{x})]\rrbracket 
        = -\chi \,  \llangle \mathrm{grad}[\vartheta^{(r)}(\mathbf{x})] \bullet \widehat{\mathbf{t}}(\mathbf{x}) \rrangle
        && \quad \mathrm{on} \; \Sigma \\
        \label{Eqn:Sensitivity_temp_jump_condition_Reverse} 
        &\llbracket\vartheta^{(r)}(\mathbf{x})\rrbracket = 0  
        && \quad \mathrm{on} \; \Sigma \\
        \label{Eqn:Sensitivity_q_BC_Reverse} 
        -&d \, \widehat{\mathbf{n}}(\mathbf{x}) \bullet \kappa(\mathbf{x}) \, \mathrm{grad}[\vartheta^{(r)}(\mathbf{x})] = 0 
        && \quad \mathrm{on} \; \partial \Omega \\
        \label{Eqn:Sensitivity_inlet_BC_Reverse} 
        &\llangle\vartheta^{(r)}(\mathbf{x})\rrangle = \vartheta_{\mathrm{inlet}} 
        = \vartheta_{\mathrm{amb}} 
        && \quad \mathrm{at} \, s = 1 \; \mathrm{on} \; \Sigma 
    \end{alignat}
\end{subequations}

Note that we have employed the same orientation for the arc-length (as the one used under the forward flow) in writing the governing equations under the reverse flow. Namely, the inlet location corresponds to $s = 0$ under the forward flow while to $s = 1$ under the reverse flow. Also, note the two main differences between the boundary value problems under forward and reverse flow conditions: the sign change on the right side of Eq.~\eqref{Eqn:Sensitivity_q_jump_condition_Reverse} (cf. Eq.~\eqref{Eqn:Sensitivity_q_jump_condition}), and $s = 1$ in Eq.~\eqref{Eqn:Sensitivity_inlet_BC_Reverse} (instead of $s = 0$ in Eq.~\eqref{Eqn:Sensitivity_inlet_BC}). 

We will refer to the boundary value problems under the forward and reverse flow conditions as the \emph{forward} and \emph{reverse flow problems}, respectively. In general, $\vartheta^{(f)}(\mathbf{x})$ differs from $\vartheta^{(r)}(\mathbf{x})$. However, as shown recently by \cite{nakshatrala2022invariance}, the mean surface temperature remains invariant under a reversal of flow when the applied heat source is uniform (i.e., $f(\mathbf{x}) = f_0$). Stated mathematically, 
\begin{align}
    \label{Eqn:Sensitivity_invariance_MST}
    &\frac{1}{\mathrm{meas}(\Omega)} \int_{\Omega} \vartheta^{(f)}(\mathbf{x}) \, \mathrm{d} \Omega 
    \; \equiv \; 
    \overline{\vartheta}^{(f)}
    = \overline{\vartheta}^{(r)} 
    \; \equiv \; 
    \frac{1}{\mathrm{meas}(\Omega)} \int_{\Omega} \vartheta^{(r)}(\mathbf{x}) \, \mathrm{d} \Omega 
\end{align}
This invariance property will be used later in the mathematical analysis of design sensitivities. 

\section{RAMIFICATIONS OF MINIMIZING MEAN SURFACE TEMPERATURE}
\label{Sec:S3_Sensitivity_Ramifications}

This section addresses the first question outlined in the introduction. We show that minimizing the mean surface temperature is equivalent to: 
\begin{enumerate}
    \item maximizing the difference between the outlet and inlet temperatures,  
    \item maximizing the outlet temperature, 
    \item maximizing the thermal (cooling) efficiency, 
    \item maximizing the difference between $\overline{\vartheta}_{\mathrm{HSS}} - \overline{\vartheta}$, and 
    \item minimizing the difference between $\overline{\vartheta} - \vartheta_{\mathrm{amb}}$ if applied heater power is uniform (i.e., $f(\mathbf{x}) = f_0$).
\end{enumerate}
Except for the last equivalence, the remaining ones hold for a general power source $f(\mathbf{x}) \geq 0$.

To establish the \emph{first} equivalence, we integrate Eq.~\eqref{Eqn:Sensitivity_BoE} over the domain, apply the divergence theorem \eqref{Eqn:Sensitivity_Divergence_theorem}, and use the rest of equations under the direct problem \eqref{Eqn:Sensitivity_q_jump_condition}--\eqref{Eqn:Sensitivity_inlet_BC} to get: 
\begin{align}
    \label{Eqn:Sensitivity_first_equivalence}
    \vartheta_{\mathrm{outlet}} - \vartheta_{\mathrm{inlet}} = \frac{1}{\chi} \left(\int_{\Omega} f(\mathbf{x}) \, \mathrm{d} \Omega - \mathrm{meas}(\Omega) \, h_{T} 
    \, \big(\overline{\vartheta} - \vartheta_{\mathrm{amb}}\big) \right)
\end{align}
In the above equation $f_0$, $h_T$, $\vartheta_{\mathrm{amb}}$, $\mathrm{meas}(\Omega)$, and $\chi$ are all independent of $\overline{\vartheta}$. Noting the negative sign in the term containing $h_T$, we conclude that minimizing the mean surface temperature will maximize the difference  $(\vartheta_{\mathrm{outlet}} - \vartheta_{\mathrm{inlet}})$. Since the inlet temperature (which is equal to the ambient temperature) is a prescribed quantity and a constant, maximizing the difference $(\vartheta_{\mathrm{outlet}} - \vartheta_{\mathrm{inlet}})$ is the same as maximizing the outlet temperature, thereby establishing the \emph{second} equivalence. The \emph{third} equivalence is evident from the definition of thermal efficiency \eqref{Eqn:Sensitivity_eta_definition_general}, which is proportional to the difference between the outlet and inlet temperatures. 

For the \emph{fourth} equivalence, we start with the definition for $\overline{\vartheta}_{\mathrm{HSS}}$ (Eq.~\eqref{Eqn:Sensitivity_Theta_HSS_definition}):
\begin{align}
    \label{Eqn:Sensitivity_Equivalence_Theta_HSS_average}
    \overline{\vartheta}_{\mathrm{HSS}} 
    = \vartheta_{\mathrm{amb}} + \frac{1}{h_T \, \mathrm{meas}(\Omega)} \int_{\Omega} f(\mathbf{x}) \, \mathrm{d} \Omega 
\end{align}
Equation \eqref{Eqn:Sensitivity_first_equivalence} can be rearranged as follows:
\begin{align}
    \label{Eqn:Sensitivity_Equivalence_Theta_average}
    \overline{\vartheta}
    = \vartheta_{\mathrm{amb}}
    + \frac{1}{h_T \, \mathrm{meas}(\Omega)} \int_{\Omega} f(\mathbf{x}) \, \mathrm{d} \Omega 
    - \frac{\chi}{h_T \, \mathrm{meas}(\Omega)} 
    \big(\vartheta_{\mathrm{outlet}} 
    - \vartheta_{\mathrm{inlet}}\big)
\end{align}
Subtracting Eq.~\eqref{Eqn:Sensitivity_Equivalence_Theta_average} from Eq.~\eqref{Eqn:Sensitivity_Equivalence_Theta_HSS_average}, we get: 
\begin{align}
    \label{Eqn:Sensitivity_Equivalence_Difference}
    \overline{\vartheta}_{\mathrm{HSS}} - 
    \overline{\vartheta}
    = \frac{\chi}{h_T \, \mathrm{meas}(\Omega)} 
    \big(\vartheta_{\mathrm{outlet}} 
    - \vartheta_{\mathrm{inlet}}\big)
\end{align}
Hence, maximizing the difference of outlet and inlet temperatures---equivalent to minimizing the mean surface temperature, from the first equivalence---implies maximizing the difference $(\overline{\vartheta}_{\mathrm{HSS}} - 
    \overline{\vartheta})$.
    
For the \emph{fifth} equivalence, we use the maximum and minimum principles recently presented by \cite{nakshatrala2022modeling}. For the case of $f(\mathbf{x}) = f_0$, the temperature field satisfies the following ordering: 
\begin{align}
    \vartheta_{\mathrm{amb}} \leq \vartheta(\mathbf{x}) \leq \overline{\vartheta}_{\mathrm{HSS}}
\end{align}
which further implies the mean surface temperature is bounded by 
\begin{align}
    \vartheta_{\mathrm{amb}} \leq \overline{\vartheta} \leq \overline{\vartheta}_{\mathrm{HSS}}
\end{align}
So, the above bounds implies that maximizing the difference $\overline{\vartheta}_{\mathrm{HSS}} - \overline{\vartheta}$ will minimize the difference $\overline{\vartheta} - \overline{\vartheta}_{\mathrm{amb}}$. This observation, alongside the fourth equivalence, will establish the desired result.

In the following two sections, we address the second question outlined in the introduction; we estimate the sensitivity of the quantity of interest (MST) to the two chosen design variables. To facilitate a pithy presentation, we avail functionals and their calculus. 

\subsection{Functionals and their calculus}
A quantity of interest (in our case, the mean surface temperature) depends on the solution field, which changes upon altering the values of the design variables. However, a solution field is \emph{not} an explicit function of design variables. Nevertheless, a sensitivity analysis should account for this solution field's non-explicit dependence. Ergo, for clarity and to ease the sensitivity analysis calculations, we introduce the ``semi-colon" notation and avail calculus of functionals. We write $\vartheta(\mathbf{x};\xi(\mathbf{x}))$ to  mean that $\vartheta$ does not explicitly depend on the quantities to the right of the semi-colon, but $\vartheta$ changes upon altering them.

In its simplest terms, a functional is a function of functions \citep{gelfand2000calculus}. So, we write a quantity of interest $\Phi$ depending on a design variable $\xi(\mathbf{x})$ as a functional of the form $\Phi[\xi(\mathbf{x})]$. Functionals have their own rich, well-established calculus. However, all we need in our sensitivity analysis is the notion of Fr\'echet derivative. For a functional $\Phi[\xi(\mathbf{x})]$, we denote its Fr\'echet derivative by $D\Phi[\xi(\mathbf{x})]$ with definition \citep{spivak2018calculus}:
\begin{align}
    \label{Eqn:Sensitivity_Frechet_definition}
    \lim_{\|\Delta\xi(\mathbf{x})\| \rightarrow 0} \;  \frac{\Phi[\xi(\mathbf{x}) + \Delta\xi(\mathbf{x})] - \Phi[\xi(\mathbf{x})] - D\Phi[\xi(\mathbf{x})] 
    \bullet \Delta \xi(\mathbf{x})}{\|\Delta\xi(\mathbf{x})\|} 
    = 0
\end{align}
where $\|\cdot\|$ is the Frobenius norm. If the functional is \emph{continuously} differentiable, then G\^ateaux variation furnishes an easier route---than the definition \eqref{Eqn:Sensitivity_Frechet_definition}---to calculate the derivative:
\begin{align}
    \left[\frac{d}{d \epsilon} \; \Phi[\xi(\mathbf{x}) + \epsilon \, \Delta \xi(\mathbf{x})] \right]_{\epsilon=0} 
    = \delta\Phi[\xi(\mathbf{x})] \bullet \Delta\xi(\mathbf{x}) 
    = D\Phi[\xi(\mathbf{x})] \bullet \Delta\xi(\mathbf{x}) 
    \quad \forall \Delta\xi(\mathbf{x})
\end{align}
implying $D\Phi[\xi(\mathbf{x})] = \delta \Phi[\xi(\mathbf{x})]$.

For easy reference, we distinguish the sensitivities to the chosen two design variables. If $\xi(\mathbf{x}) = \chi$ (i.e., the design variable is the heat capacity rate), we denote the corresponding Fr\'echet derivative as 
\begin{align}
    \Phi^{\#}[\chi] \equiv D\Phi[\chi]
\end{align}
Likewise, if $\xi(\mathbf{x}) = \kappa(\mathbf{x})$, we use the notation:  
\begin{align}
    \Phi^{\prime}[\kappa(\mathbf{x})] \equiv D\Phi[\kappa(\mathbf{x})]
\end{align}

For the scalar solution field $\vartheta(\mathbf{x};\xi(\mathbf{x}))$ (i.e., the temperature field satisfying the forward problem), we define $D_1 \vartheta$ and $D_2\vartheta$ as follows:
\begin{align}
    &\lim_{\|\Delta\mathbf{x}\| \rightarrow 0} \;  \frac{\vartheta\big(\mathbf{x}+\Delta \mathbf{x};\xi(\mathbf{x})\big) 
    - \vartheta\big(\mathbf{x};\xi(\mathbf{x})\big)
    - D_{1}\vartheta\big(\mathbf{x};\xi(\mathbf{x})\big)
    \bullet \Delta \mathbf{x}}{\|\Delta\mathbf{x}\|} 
    = 0 \\
    &\lim_{\|\Delta\xi(\mathbf{x})\| \rightarrow 0} \;  \frac{\vartheta\big(\mathbf{x};\xi(\mathbf{x})+ \Delta \xi(\mathbf{x})\big) 
    - \vartheta\big(\mathbf{x};\xi(\mathbf{x})\big)
    - D_{2}\vartheta\big(\mathbf{x};\xi(\mathbf{x})\big)
    \bullet \Delta \xi(\mathbf{x})}{\|\Delta\xi(\mathbf{x})\|} 
    = 0
\end{align}
Then the spatial derivative is related as follows: 
\begin{align}
    \mathrm{grad}\big[\vartheta\big(\mathbf{x};\xi(\mathbf{x})\big)\big] = D_{1} \vartheta\big(\mathbf{x};\xi(\mathbf{x})\big)
\end{align}
Similar to the notation used for the functional $\Phi$, we use the following notation to represent the solution sensitivities:
\begin{align}
    \vartheta^{\#}\big(\mathbf{x};\chi\big) = D_2\vartheta\big(\mathbf{x};\chi\big) 
    \quad \mathrm{and} \quad
    \vartheta^{\prime}\big(\mathbf{x};\kappa(\mathbf{x})\big) = D_2\vartheta\big(\mathbf{x};\kappa(\mathbf{x})\big) 
\end{align}

In the next section, we mathematically investigate the sensitivity of the mean surface temperature to the fluid's heat capacity rate. 

\section{SENSITIVITY OF MST TO COOLANT'S HEAT CAPACITY RATE}
\label{Sec:S4_Sensitivity_Heat_capacity_rate}
For this set of sensitivity analysis, we write the objective functional as follows:
\begin{align}
    \label{Eqn:Sensitivity_Phi_for_chi}
    \Phi[\chi] 
    &= \int_{\Omega} 
    \vartheta\big(\mathbf{x};\chi\big) \, \mathrm{d} \Omega 
\end{align}
Since, in this case, $\Phi = \mathrm{meas}(\Omega) \, \overline{\vartheta}$ and $\mathrm{meas}(\Omega)$---the area of the domain---is a constant, minimizing $\Phi$ is \emph{equivalent} to minimizing the mean surface temperature. Our usage of this alternative (but equivalent) objective functional is for mathematical convenience: to avoid writing $1/\mathrm{meas}(\Omega)$ factor often.

To find the associated design sensitivity, the task at hand is to calculate $D \Phi[\chi]$:
\begin{align}
    \label{Eqn:Sensitivity_DPhi_chi}
    D \Phi[\chi] 
    = \int_{\Omega} 
    \vartheta^{\#}\big(\mathbf{x};\chi\big) \, \mathrm{d} \Omega 
\end{align}
where a superscript $\#$ denotes the Fr\'echet derivative with respect to $\chi$. Notice that the above expression for the design sensitivity contains solution sensitivity $\vartheta^{\#}(\mathbf{x};\chi)$. But, the solution is unknown without solving the direct problem; the same situation even with the solution sensitivity. It will be ideal if we can estimate  $D\Phi[\chi]$ without actually finding $\vartheta^{\#}(\mathbf{x};\chi)$. The \emph{adjoint state method} provides one such viable route, and we will avail it below.

We augment Eq.~\eqref{Eqn:Sensitivity_DPhi_chi} with terms involving products of a Lagrange multiplier and derivatives with respect to $\chi$ of the residuals of state equations (i.e., governing equations of the direct problem). After augmenting these terms, the design sensitivity can be equivalently written as follows:
\begin{align}
    \label{Eqn:Sensitivity_Expression_chi_original_terms}
    D \Phi[\chi] 
    &= \int_{\Omega} \vartheta^{\#}(\mathbf{x};\chi) \, \mathrm{d} \Omega \notag \\
    & \quad \quad \quad +\frac{1}{f_0} \int_{\Omega \setminus \Sigma}
    \Big(\mu(\mathbf{x}) - \vartheta_{\mathrm{amb}}\Big) \,  \Big(\underbrace{d \, \mathrm{div}[\kappa(\mathbf{x})\mathrm{grad}[\vartheta(\mathbf{x};\chi)]] + f_{0} - h_{T}\big(\vartheta(\mathbf{x};\chi) - \vartheta_{\mathrm{amb}}\big)}_{Eq.~\eqref{Eqn:Sensitivity_BoE}}\Big)^{\#} \mathrm{d} \Omega \notag \\
    &\quad \quad \quad -\frac{1}{f_0} \int_{\Sigma}  \Big\llangle\mu(\mathbf{x}) -  \vartheta_{\mathrm{amb}}\Big\rrangle \, \Big(\underbrace{\big\llbracket d \, \kappa(\mathbf{x}) \mathrm{grad}[\vartheta(\mathbf{x};\chi)]\big\rrbracket + \chi \, \llangle \mathrm{grad}[\vartheta(\mathbf{x};\chi)]\bullet \widehat{\mathbf{t}}(\mathbf{x})\rrangle}_{Eq.~\eqref{Eqn:Sensitivity_q_jump_condition}}\Big)^{\#} \mathrm{d} \Gamma \notag \\
    &\quad \quad \quad +\frac{1}{f_0} \int_{\Sigma} \Big\llangle d \, \kappa(\mathbf{x}) \mathrm{grad}[\mu(\mathbf{x})]\Big\rrangle
    \bullet 
    \Big(\underbrace{ 
    \llbracket\vartheta(\mathbf{x};\chi)\rrbracket}_{Eq.~\eqref{Eqn:Sensitivity_temp_jump_condition}}\Big)^{\#} \, \mathrm{d} \Gamma \notag \\
    &\quad \quad \quad +\frac{1}{f_0} \int_{\partial \Omega} \Big(\mu(\mathbf{x}) -  \vartheta_{\mathrm{amb}}\Big) \, \Big(\underbrace{-d \, \kappa(\mathbf{x}) \mathrm{grad}[\vartheta(\mathbf{x};\chi)] \bullet \widehat{\mathbf{n}}(\mathbf{x})}_{Eq.~\eqref{Eqn:Sensitivity_q_BC}}\Big)^{\#} \, \mathrm{d} \Gamma \notag \\
    &\quad \quad \quad -\frac{\chi}{f_0} \Big(\llangle\mu(\mathbf{x})\rrangle - \vartheta_{\mathrm{amb}}\Big) \,  \Big(\underbrace{\llangle\vartheta(\mathbf{x};\chi) \rrangle - \vartheta_{\mathrm{amb}}}_{Eq.~\eqref{Eqn:Sensitivity_inlet_BC}}\Big)^{\#} \Big|_{s = 0 \;  (\mathrm{inlet})}
\end{align}
where $\mu(\mathbf{x})$ is the newly introduced Lagrange multiplier---also known as the adjoint variable. In the above equation, the factors introduced in the augmented terms are for getting an easy-to-work adjoint state problem, which will be apparent later in the derivation. 

After applying Green's theorem twice, using identities \eqref{Eqn:Sensitivities_avg_jump_identities}, and grouping the terms, Eq.~\eqref{Eqn:Sensitivity_Expression_chi_original_terms} can be rewritten as follows (see Appendix for a detailed derivation): 
\begin{align}
    \label{Eqn:Sensitivity_Expression_chi_collected_terms}
    D \Phi[\chi] 
    &= - \frac{1}{f_0} \int_{\Sigma} \Big(\big\llangle\mu(\mathbf{x})\big\rrangle - \vartheta_{\mathrm{amb}} \Big)
    \, \big\llangle
    \mathrm{grad}[\vartheta(\mathbf{x};\chi)] \bullet \widehat{\mathbf{t}}(\mathbf{x}) 
    \big\rrangle 
    \, \mathrm{d} \Gamma \notag \\ 
    &\qquad +  \frac{1}{f_0} \int_{\Omega\setminus\Sigma}
    \vartheta^{\#}(\mathbf{x};\chi) 
    \, 
    {\color{purple}
    \Big\{d \, \mathrm{div}[\kappa(\mathbf{x})\mathrm{grad}[\mu(\mathbf{x})]] + f_{0} - h_{T}\big(\mu(\mathbf{x}) - \vartheta_{\mathrm{amb}}\big)\Big\}
    \, \mathrm{d} \Omega
    }
    \notag \\
    &\qquad - \frac{1}{f_0} \int_{\Sigma} \big\llangle \vartheta^{\#}(\mathbf{x};\chi)
    \big\rrangle \, 
    {\color{purple}
    \Big\{
    \big\llbracket d \, \kappa(\mathbf{x}) \mathrm{grad}[\mu(\mathbf{x})] \big\rrbracket 
    - \chi \, \big\llangle \mathrm{grad}[\mu(\mathbf{x})] \bullet
    \widehat{\mathbf{t}}(\mathbf{x}) \big\rrangle 
    \Big\}
    \, \mathrm{d} \Gamma
    }
    \notag \\
    &\qquad + \frac{1}{f_0} \int_{\Sigma} \big\llangle d \, \kappa(\mathbf{x}) \mathrm{grad}[\vartheta^{\#}(\mathbf{x};\chi)]\big\rrangle \bullet 
    {\color{purple}
    \Big\{\llbracket \mu(\mathbf{x}) - \vartheta_{\mathrm{amb}}\rrbracket \Big\} 
    \, \mathrm{d} \Gamma
    }
    \notag \\
    &\qquad + \frac{1}{f_0} \int_{\partial \Omega} \vartheta^{\#}(\mathbf{x};\chi) \, 
    {\color{purple}
    \Big\{-d \, 
    \kappa(\mathbf{x}) \mathrm{grad}[\mu(\mathbf{x})] 
    \bullet \widehat{\mathbf{n}}(\mathbf{x}) \Big\} 
    \, \mathrm{d} \Gamma
    }
    \notag \\
    &\qquad 
    - \frac{\chi}{f_0} \big\llangle \vartheta^{\#}(\mathbf{x};\chi) 
    \big\rrangle \, 
    {\color{purple}
    \Big\{\big\llangle\mu(\mathbf{x})\big\rrangle - \vartheta_{\mathrm{amb}}\Big\} \Big\vert_{s = 1} 
    }
\end{align}
Note that $\mu(\mathbf{x})$ is arbitrary till now, and a prudent choice for it will simplify the design sensitivity expression. A heedful look at Eq.~\eqref{Eqn:Sensitivity_Expression_chi_collected_terms} reveals that if we set all the terms in the curly brackets\footnote{We also marked these terms in purple for the reader's benefit---see an online version of this article.} to zero, the resulting expression for the design sensitivity will not contain $\vartheta^{\#}(\mathbf{x};\chi)$. 

Duly, we force the mentioned terms to vanish, giving rise to the following boundary value problem: 
\begin{subequations}
\begin{alignat}{2}
    \label{Eqn:Sensitivity_BoE_AP_chi} 
    -&d \, \mathrm{div}[\kappa(\mathbf{x})\mathrm{grad}[\mu(\mathbf{x})]] = f_{0}
    - h_{T} \,  \big(\mu(\mathbf{x}) - \vartheta_{\mathrm{amb}}\big) 
    && \quad \mathrm{in} \; \Omega \setminus \Sigma \\
    \label{Eqn:Sensitivity_q_jump_condition_AP_chi} 
    -&d \, \llbracket\kappa(\mathbf{x}) \mathrm{grad}[\mu(\mathbf{x})]\rrbracket = -\chi \, \llangle \mathrm{grad}[\mu(\mathbf{x})] \bullet \widehat{\mathbf{t}}(\mathbf{x}) \rrangle  
    && \quad \mathrm{on} \; \Sigma \\
    \label{Eqn:Sensitivity_temp_jump_condition_AP_chi} 
    &\llbracket\mu(\mathbf{x}) \rrbracket = 0  
    && \quad \mathrm{on} \; \Sigma \\
    \label{Eqn:Sensitivity_q_BC_AP_chi} 
    -&d \, \widehat{\mathbf{n}}(\mathbf{x}) \bullet \kappa(\mathbf{x})\mathrm{grad}[\mu(\mathbf{x})] = 0 
    && \quad \mathrm{on} \; \partial \Omega \\
    \label{Eqn:Sensitivity_inlet_BC_AP_chi} 
    &\llangle \mu(\mathbf{x}) \rrangle = \vartheta_{\mathrm{amb}} 
    && \quad \mathrm{at} \, s = 1 \, \mathrm{on} \, \Sigma 
\end{alignat}
\end{subequations}
The above boundary value problem is referred to as the adjoint state problem (or adjoint problem, for brevity). If $\mu(\mathbf{x})$ satisfies the above adjoint state problem, the design sensitivity takes the following compact form:
\begin{align}
    \label{Eqn:ROM_Sensitivity_wrt_chi}
    D\Phi[\chi] 
    = - \frac{1}{f_0} \int_{\Sigma} \Big(\big\llangle\mu(\mathbf{x}) 
    \big \rrangle- \vartheta_{\mathrm{amb}} \Big) \big\llangle \mathrm{grad}[\vartheta(\mathbf{x};\chi)] \bullet \widehat{\mathbf{t}}(\mathbf{x}) 
    \big \rrangle \, \mathrm{d} \Gamma 
\end{align}

Recalling the discussion in \S\ref{Subsec:Sensitivity_forward_reverse_flows}, we note that the above adjoint state problem is identical to the boundary value problem under reverse flow conditions (cf.~Eqs.~\eqref{Eqn:Sensitivity_BoE_Reverse}--\eqref{Eqn:Sensitivity_inlet_BC_Reverse}), 
which is s well-posed problem with a unique solution \citep{nakshatrala2022modeling}. Thus, the solution to the adjoint variable is 
\begin{align}
    \mu(\mathbf{x}) = \vartheta^{(r)}(\mathbf{x}) 
\end{align}
With this identification, invoking the continuity of $\mu(\mathbf{x})$ across the vasculature \eqref{Eqn:Sensitivity_temp_jump_condition_AP_chi}, and noting Remark \ref{Remark:Sensitivity_Tangential_component}, the sensitivity of the \emph{mean surface temperature} to the \emph{fluid's heat capacity rate} amounts to
\begin{align}
    \label{Eqn:ROM_Sensitivity_wrt_chi_modified}
    D\Phi[\chi] 
    &= - \frac{1}{f_0} \int_{\Sigma} \left(\vartheta^{(r)}(\mathbf{x}) - \vartheta_{\mathrm{amb}} \right) \mathrm{grad}\Big[\vartheta^{(f)}(\mathbf{x})\Big] \bullet \widehat{\mathbf{t}}(\mathbf{x}) \, \mathrm{d} \Gamma 
\end{align}
where $\vartheta^{(f)}(\mathbf{x})$ is the solution of the boundary value problem under the forward flow conditions---the solution of the direct problem (i.e.,  Eqs.~\eqref{Eqn:Sensitivity_BoE}--\eqref{Eqn:Sensitivity_inlet_BC}). 

The exact expressions for $\vartheta^{(f)}(\mathbf{x})$ and $\vartheta^{(r)}(\mathbf{x})$ are \emph{un}known \emph{a priori}. One often solves the direct and adjoint boundary value problems to get these solutions using a numerical method or an analytical technique. Since we are interested only in the design sensitivity's sign, we do not attempt to solve these problems but avail mathematical analysis to assess the remarked sign instead.

\begin{theorem}
    \label{Thm:Sensitivity_DPhiChi}
    Under uniform $f(\mathbf{x}) = f_0$, the sensitivity of the mean surface temperature to the heat capacity rate is \emph{non-positive}. That is, 
    \begin{align}
        D\Phi[\chi] 
        = - \frac{1}{f_0} \int_{\Sigma} \left(\vartheta^{(r)}(\mathbf{x}) - \vartheta_{\mathrm{amb}} \right) \mathrm{grad}\Big[\vartheta^{(f)}(\mathbf{x})\Big] \bullet \widehat{\mathbf{t}}(\mathbf{x}) \, \mathrm{d} \Gamma \leq 0
    \end{align}
\end{theorem} 
\begin{proof}
By multiplying the first equation of the forward flow problem by $(\vartheta^{(f)}(\mathbf{x}) - \vartheta_{\mathrm{amb}})$, applying Green's identity, and using the rest of the equations, we write  
\begin{align}
    \label{Eqn:Sensitivity_chi_step1}
    \int_{\Omega} d \, \mathrm{grad}\left[\vartheta^{(f)}(\mathbf{x})\right] \bullet \kappa(\mathbf{x}) \mathrm{grad}\left[\vartheta^{(f)}(\mathbf{x})\right] \, \mathrm{d} \Omega
    +\int_{\Sigma} \chi \, 
    \left(\vartheta^{(f)}(\mathbf{x}) - \vartheta_{\mathrm{amb}}\right)
    \mathrm{grad}\left[\vartheta^{(f)}(\mathbf{x})\right] \bullet
    \widehat{\mathbf{t}}(\mathbf{x}) \, \mathrm{d} \Gamma \notag \\
    =\int_{\Omega} f(\mathbf{x}) \left(\vartheta^{(f)}(\mathbf{x}) - \vartheta_{\mathrm{amb}}\right) \mathrm{d} \Omega
    -\int_{\Omega} h_{T} \left(\vartheta^{(f)}(\mathbf{x}) - \vartheta_{\mathrm{amb}}\right)^2 \mathrm{d} \Omega
\end{align}
Likewise, by multiplying the first equation of the reverse flow problem by $(\vartheta^{(r)}(\mathbf{x}) - \vartheta_{\mathrm{amb}})$ and following similar steps as before, we get the following: 
\begin{align}
    \label{Eqn:Sensitivity_chi_step2}
    \int_{\Omega} d \, \mathrm{grad}\left[\vartheta^{(r)}(\mathbf{x})\right] \bullet \kappa(\mathbf{x}) \mathrm{grad}\left[\vartheta^{(r)}(\mathbf{x})\right] \, \mathrm{d} \Omega
    -\int_{\Sigma} \chi \, 
    \left(\vartheta^{(r)}(\mathbf{x}) - \vartheta_{\mathrm{amb}}\right)
    \mathrm{grad}\left[\vartheta^{(r)}(\mathbf{x})\right] \bullet
    \widehat{\mathbf{t}}(\mathbf{x}) \, \mathrm{d} \Gamma \notag \\
    =\int_{\Omega} f(\mathbf{x}) \left(\vartheta^{(r)}(\mathbf{x}) - \vartheta_{\mathrm{amb}}\right) \mathrm{d} \Omega
    -\int_{\Omega} h_{T} \left(\vartheta^{(r)}(\mathbf{x}) - \vartheta_{\mathrm{amb}}\right)^2 \mathrm{d} \Omega
\end{align}
Finally, by multiplying the first equation of the forward flow problem by $(\vartheta^{(r)}(\mathbf{x}) - \vartheta_{\mathrm{amb}})$, we get the following:
\begin{align}
    \label{Eqn:Sensitivity_chi_step3}
    \int_{\Omega} d \, \mathrm{grad}\left[\vartheta^{(r)}(\mathbf{x})\right] \bullet \kappa(\mathbf{x}) \mathrm{grad}\left[\vartheta^{(f)}(\mathbf{x})\right] \, \mathrm{d} \Omega
    +\int_{\Sigma} \chi \, 
    \left(\vartheta^{(r)}(\mathbf{x}) - \vartheta_{\mathrm{amb}}\right)
    \mathrm{grad}\left[\vartheta^{(f)}(\mathbf{x})\right] \bullet 
    \widehat{\mathbf{t}}(\mathbf{x}) \, \mathrm{d} \Gamma \notag \\
    =\int_{\Omega} f(\mathbf{x}) \left(\vartheta^{(r)}(\mathbf{x}) - \vartheta_{\mathrm{amb}}\right) \mathrm{d} \Omega
    -\int_{\Omega} h_{T} \left(\vartheta^{(r)}(\mathbf{x}) - \vartheta_{\mathrm{amb}}\right) 
    \left(\vartheta^{(f)}(\mathbf{x}) - \vartheta_{\mathrm{amb}}\right) \mathrm{d} \Omega
\end{align}
We add Eqs.~\eqref{Eqn:Sensitivity_chi_step1} and \eqref{Eqn:Sensitivity_chi_step2}, and from this sum we subtract  twice Eq.~\eqref{Eqn:Sensitivity_chi_step3} (i.e., Eq.~\eqref{Eqn:Sensitivity_chi_step1} + Eq.~\eqref{Eqn:Sensitivity_chi_step2} -- 2 $\times$ Eq.~\eqref{Eqn:Sensitivity_chi_step3}); this calculation amounts to
\begin{align}
    \label{Eqn:Sensitivity_chi_step4}
    &\int_{\Omega} d \, \mathrm{grad}\left[\vartheta^{(f)}(\mathbf{x}) - \vartheta^{(r)}(\mathbf{x})\right] \bullet \kappa(\mathbf{x}) \mathrm{grad}\left[\vartheta^{(f)}(\mathbf{x}) - \vartheta^{(r)}(\mathbf{x})\right] \, \mathrm{d} \Omega
    +\int_{\Omega} h_{T} \left(\vartheta^{(f)}(\mathbf{x}) - 
    \vartheta^{(r)}(\mathbf{x})\right)^{2} 
    \mathrm{d} \Omega \notag \\
    &\qquad \qquad = \int_{\Omega} f(\mathbf{x}) \left(\vartheta^{(f)}(\mathbf{x}) - \vartheta^{(r)}(\mathbf{x})\right) \mathrm{d} \Omega \notag \\
    &\qquad \qquad \qquad -\int_{\Sigma} \chi \, 
    \left(\vartheta^{(f)}(\mathbf{x}) - \vartheta_{\mathrm{amb}}\right)
    \mathrm{grad}\left[\vartheta^{(f)}(\mathbf{x})\right] \bullet
    \widehat{\mathbf{t}}(\mathbf{x}) \, \mathrm{d} \Gamma \notag \\
    &\qquad \qquad \qquad +\int_{\Sigma} \chi \, 
    \left(\vartheta^{(r)}(\mathbf{x}) - \vartheta_{\mathrm{amb}}\right)
    \mathrm{grad}\left[\vartheta^{(r)}(\mathbf{x})\right] \bullet
    \widehat{\mathbf{t}}(\mathbf{x}) \, \mathrm{d}
    \Gamma \notag \\
    &\qquad \qquad \qquad +2\int_{\Sigma} \chi \, 
    \left(\vartheta^{(r)}(\mathbf{x}) - \vartheta_{\mathrm{amb}}\right)
    \mathrm{grad}\left[\vartheta^{(f)}(\mathbf{x})\right] \bullet 
    \widehat{\mathbf{t}}(\mathbf{x}) \, \mathrm{d} \Gamma
\end{align}

The two integrals on the left side of the above equation are non-negative; note $d > 0$ and $\kappa(\mathbf{x}) > 0$. Therefore, we write
\begin{align}
    \label{Eqn:Sensitivity_chi_step5}
    &0 \leq \int_{\Omega} f(\mathbf{x}) \left(\vartheta^{(f)}(\mathbf{x}) - \vartheta^{(r)}(\mathbf{x})\right) 
    \mathrm{d} \Omega \notag \\
    &\qquad \qquad \qquad -\int_{\Sigma} \chi \, 
    \left(\vartheta^{(f)}(\mathbf{x}) - \vartheta_{\mathrm{amb}}\right)
    \mathrm{grad}\left[\vartheta^{(f)}(\mathbf{x})\right] \bullet 
    \widehat{\mathbf{t}}(\mathbf{x}) \, \mathrm{d} \Gamma \notag \\
    &\qquad \qquad \qquad +\int_{\Sigma} \chi \, 
    \left(\vartheta^{(r)}(\mathbf{x}) - \vartheta_{\mathrm{amb}}\right)
    \mathrm{grad}\left[\vartheta^{(r)}(\mathbf{x})\right] \bullet 
    \widehat{\mathbf{t}}(\mathbf{x}) \, \mathrm{d}
    \Gamma \notag \\
    &\qquad \qquad \qquad 
    +2\int_{\Sigma} \chi \, 
    \left(\vartheta^{(r)}(\mathbf{x}) - \vartheta_{\mathrm{amb}}\right)
    \mathrm{grad}\left[\vartheta^{(f)}(\mathbf{x})\right] \bullet
    \widehat{\mathbf{t}}(\mathbf{x}) \, \mathrm{d} \Gamma
\end{align}
By noting $f(\mathbf{x})$ is uniform and invoking the invariance of the mean surface temperature under flow reversal (i.e., Eq.~\eqref{Eqn:Sensitivity_invariance_MST}), the first term on the right side of Eq.~\eqref{Eqn:Sensitivity_chi_step5} vanishes. We thus have the following inequality: 
\begin{align}
    \label{Eqn:Sensitivity_chi_step6}
    \int_{\Sigma} \chi \, 
    \left(\vartheta^{(f)}(\mathbf{x}) - \vartheta_{\mathrm{amb}}\right)
    \mathrm{grad}\left[\vartheta^{(f)}(\mathbf{x})\right] \bullet 
    \widehat{\mathbf{t}}(\mathbf{x}) \, \mathrm{d} \Gamma 
    -\int_{\Sigma} \chi \, 
    \left(\vartheta^{(r)}(\mathbf{x}) - \vartheta_{\mathrm{amb}}\right)
    \mathrm{grad}\left[\vartheta^{(r)}(\mathbf{x})\right] \bullet 
    \widehat{\mathbf{t}}(\mathbf{x}) \, \mathrm{d}
    \Gamma \notag \\
    \qquad \leq 2\int_{\Sigma} \chi \, 
    \left(\vartheta^{(r)}(\mathbf{x}) - \vartheta_{\mathrm{amb}}\right)
    \mathrm{grad}\left[\vartheta^{(f)}(\mathbf{x})\right] \bullet 
    \widehat{\mathbf{t}}(\mathbf{x}) \, \mathrm{d} \Gamma
\end{align}

We will now consider the first term in Eq.~\eqref{Eqn:Sensitivity_chi_step6} and execute the integral along the vasculature:
\begin{align}
    \label{Eqn:Sensitivity_chi_step7}
    \int_{\Sigma} \chi \, 
    \left(\vartheta^{(f)}(\mathbf{x}) - \vartheta_{\mathrm{amb}}\right)
    \mathrm{grad}\left[\vartheta^{(f)}(\mathbf{x})\right] \bullet
    \widehat{\mathbf{t}}(\mathbf{x}) \, \mathrm{d} \Gamma 
    = \frac{\chi}{2} \, 
    \left(\vartheta^{(f)}(\mathbf{x}) - \vartheta_{\mathrm{amb}}\right)^2 \Big|_{s = 0}^{s = 1}
\end{align}
Noting that, under forward flow conditions, $\vartheta^{(f)}(\mathbf{x}) = \vartheta_{\mathrm{amb}}$ at $s = 0$ and $\vartheta^{(f)}(\mathbf{x}) = \vartheta^{(f)}_{\mathrm{outlet}}$ at $s = 1$, we get
\begin{align}
    \label{Eqn:Sensitivity_chi_step8}
    \int_{\Sigma} \chi \, 
    \left(\vartheta^{(f)}(\mathbf{x}) - \vartheta_{\mathrm{amb}}\right)
    \mathrm{grad}\left[\vartheta^{(f)}(\mathbf{x})\right] \bullet
    \widehat{\mathbf{t}}(\mathbf{x}) \, \mathrm{d} \Gamma 
    = \frac{\chi}{2} \, 
    \left(\vartheta^{(f)}_{\mathrm{outlet}} - \vartheta_{\mathrm{amb}}\right)^2 
\end{align}
Carrying out similar steps for the second term in Eq.~\eqref{Eqn:Sensitivity_chi_step6}, we write 
\begin{align}
    \label{Eqn:Sensitivity_chi_step9}
    \int_{\Sigma} \chi \, 
    \left(\vartheta^{(r)}(\mathbf{x}) - \vartheta_{\mathrm{amb}}\right)
    \mathrm{grad}\left[\vartheta^{(r)}(\mathbf{x})\right] \bullet
    \widehat{\mathbf{t}}(\mathbf{x}) \, \mathrm{d} \Gamma 
    &= \frac{\chi}{2} \, 
    \left(\vartheta^{(r)}(\mathbf{x}) - \vartheta_{\mathrm{amb}}\right)^2 \Big|_{s = 0}^{s = 1}
    \notag \\
    &= -\frac{\chi}{2} \, 
    \left(\vartheta^{(r)}_{\mathrm{outlet}} - \vartheta_{\mathrm{amb}}\right)^2 
\end{align}
The minus sign arises because, under reverse flow conditions, $\vartheta^{(r)}(\mathbf{x}) = \vartheta^{(r)}_{\mathrm{outlet}}$ at $s = 0$ and $\vartheta^{(f)}(\mathbf{x}) = \vartheta_{\mathrm{amb}}$ at $s = 1$. Using Eqs.~\eqref{Eqn:Sensitivity_chi_step7} and \eqref{Eqn:Sensitivity_chi_step8}, inequality \eqref{Eqn:Sensitivity_chi_step9} can be written as follows: 
\begin{align}
    \label{Eqn:Sensitivity_chi_step10}
    \frac{\chi}{2} 
    \left(\vartheta^{(f)}_{\mathrm{outlet}} - \vartheta_{\mathrm{amb}}\right)^{2}
    +\frac{\chi}{2} 
    \left(\vartheta^{(r)}_{\mathrm{outlet}} - \vartheta_{\mathrm{amb}}\right)^{2}
    \leq 2\int_{\Sigma} \chi \, 
    \left(\vartheta^{(r)}(\mathbf{x}) - \vartheta_{\mathrm{amb}}\right)
    \mathrm{grad}\left[\vartheta^{(f)}(\mathbf{x})\right] \bullet
    \widehat{\mathbf{t}}(\mathbf{x}) \, \mathrm{d} \Gamma
\end{align}
The left side of the above inequality is non-negative. Hence, we have 
\begin{align}
    \label{Eqn:Sensitivity_chi_step11}
    0 \leq 
    2\int_{\Sigma} \chi \, 
    \left(\vartheta^{(r)}(\mathbf{x}) - \vartheta_{\mathrm{amb}}\right)
    \mathrm{grad}\left[\vartheta^{(f)}(\mathbf{x})\right] \bullet 
    \widehat{\mathbf{t}}(\mathbf{x}) \, \mathrm{d} \Gamma
\end{align}
Noting that $\chi \geq 0$ and $f_0 > 0$ are constants, the above inequality renders the desired result:
\begin{align}
    \label{Eqn:Sensitivity_chi_step12}
    D\Phi[\chi] 
    = - \frac{1}{f_0} \int_{\Sigma} \left(\vartheta^{(r)}(\mathbf{x}) - \vartheta_{\mathrm{amb}} \right) \mathrm{grad}\Big[\vartheta^{(f)}(\mathbf{x})\Big] \bullet \widehat{\mathbf{t}}(\mathbf{x}) \, \mathrm{d} \Gamma \leq 0
\end{align}
\end{proof} 

\subsection{Discussion} The above result is remarkable: We have mathematically shown that, regardless of the vasculature, the mean surface temperature always decreases upon increasing the heat capacity rate. Even if the segments of the vasculature are close by---which introduces countercurrent heat exchange: heat transfers from the coolant to the host solid---the remarked trend is unaltered. Given the heat capacity rate is a product of fluid properties (specific heat capacity, $c_f$, and density, $\rho_f$) and volumetric flow rate ($Q$), two scenarios are pertinent: one can alter the fluid or vary the flow rate.
\begin{enumerate}
    \item For a fixed fluid, increasing the volumetric flow rate will decrease the mean surface temperature. 
    \item For a fixed flow rate while altering the flowing fluid (coolant), the fluid with a higher heat capacity (i.e., product of the density and specific heat capacity: $\rho_f  c_f$) will have a lower mean surface temperature. 
\end{enumerate}

\section{SENSITIVITY OF MST TO THERMAL CONDUCTIVITY}
\label{Sec:S5_Sensitivity_Conductivity}
For estimating the sensitivity of the mean surface temperature to the host material's thermal conductivity field $\kappa(\mathbf{x})$, the task is to find $D \Phi[\kappa(\mathbf{x})]$. We again use the adjoint state method and start with the definition of $D\Phi[\kappa]$: 
\begin{align}
    \label{Eqn:Sensitivity_design_sensitivity_kappa}
    D \Phi[\kappa(\mathbf{x})] 
    &= \int_{\Omega} 
    \vartheta^{\prime}\big(\mathbf{x};\kappa(\mathbf{x})\big) \, \mathrm{d} \Omega 
\end{align}
where a superscript prime denotes the Fr\'echet derivative related to the  conductivity field. 

Following the steps taken in the previous section, we augment Eq.~\eqref{Eqn:Sensitivity_design_sensitivity_kappa} with terms involving weighted integrals, containing the derivatives with respect to $\kappa(\mathbf{x})$ of the residuals of the forward problem's governing equations. $\lambda(\mathbf{x})$ will now denote the corresponding weight (i.e., the Lagrange multiplier or the adjoint variable). After augmenting these terms, the design sensitivity can be written as follows: 
\begin{align}
D \Phi[k(\mathbf{x})] 
&= \int_{\Omega} \vartheta^{\prime}(\mathbf{x};\kappa(\mathbf{x})) \, \mathrm{d} \Omega \notag \\
& \quad \quad \quad +\frac{1}{f_0} \int_{\Omega}
\Big(\lambda(\mathbf{x}) - \vartheta_{\mathrm{amb}}\Big) \;  \Big(\underbrace{d \, \mathrm{div}[\kappa(\mathbf{x}) \, \mathrm{grad}[\vartheta(\mathbf{x};\kappa(\mathbf{x}))]] + f_{0} - h_{T} \, \big(\vartheta(\mathbf{x};\kappa(\mathbf{x})) - \vartheta_{\mathrm{amb}}\big)}_{Eq.~\eqref{Eqn:Sensitivity_BoE}}\Big)^{\prime} \mathrm{d} \Omega \notag \\
&\quad \quad \quad -\frac{1}{f_0} \int_{\Sigma} \Big\llangle\lambda(\mathbf{x}) -  \vartheta_{\mathrm{amb}}\Big\rrangle \; \Big(\underbrace{\llbracket d \, \kappa(\mathbf{x}) \mathrm{grad}[\vartheta(\mathbf{x};\kappa(\mathbf{x}))]\rrbracket + \chi \, 
\big\llangle \mathrm{grad}[\vartheta(\mathbf{x};\kappa(\mathbf{x}))]\bullet \widehat{\mathbf{t}}(\mathbf{x})\big\rrangle}_{Eq.~\eqref{Eqn:Sensitivity_q_jump_condition}}\Big)^{\prime} \mathrm{d} \Gamma \notag \\
&\quad \quad \quad +\frac{1}{f_0} \int_{\Sigma} \Big\llangle d \, \kappa(\mathbf{x}) \mathrm{grad}[\lambda(\mathbf{x})]\Big\rrangle
\bullet 
\Big(\underbrace{ 
\llbracket\vartheta(\mathbf{x};\kappa(\mathbf{x}))\rrbracket}_{Eq.~\eqref{Eqn:Sensitivity_temp_jump_condition}}\Big)^{\prime} \, \mathrm{d} \Gamma \notag \\
&\quad \quad \quad +\frac{1}{f_0} \int_{\partial \Omega} \Big(\lambda(\mathbf{x}) -  \vartheta_{\mathrm{amb}}\Big) \, \Big(\underbrace{-d \, \kappa(\mathbf{x}) \mathrm{grad}[\vartheta(\mathbf{x};\kappa(\mathbf{x}))] \bullet \widehat{\mathbf{n}}(\mathbf{x})}_{Eq.~\eqref{Eqn:Sensitivity_q_BC}}\Big)^{\prime} \, \mathrm{d} \Gamma \notag \\
&\quad \quad \quad -\frac{\chi}{f_0} \Big(\big\llangle\lambda(\mathbf{x})\big\rrangle - \vartheta_{\mathrm{amb}}\Big) \,  \Big(\underbrace{\big\llangle\vartheta(\mathbf{x};\kappa(\mathbf{x}))\big\rrangle - \vartheta_{\mathrm{amb}}}_{Eq.~\eqref{Eqn:Sensitivity_inlet_BC}}\Big)^{\prime} \Big|_{\mathrm{inlet}}
\end{align}
Similar to the derivation provided in the previous section and appendix, the above expression can be rewritten as follows:
\begin{align}
\label{Eqn:Sentivities_Expression_collected_terms}
D \Phi[k(\mathbf{x})] 
&= -\frac{1}{f_0} \int_{\Omega} d \, 
\mathrm{grad}[\lambda(\mathbf{x})] \bullet  
\mathrm{grad}[\vartheta(\mathbf{x};\kappa(\mathbf{x}))] \, \mathrm{d} \Omega  \notag \\ 
&\qquad +  \frac{1}{f_0} \int_{\Omega}
\vartheta^{\prime}(\mathbf{x};\kappa(\mathbf{x})) 
\, 
{\color{purple}
\Big\{d \, \mathrm{div}[\kappa(\mathbf{x}) \, \mathrm{grad}[\lambda(\mathbf{x})]] + f_{0} - h_{T} \, \big(\lambda(\mathbf{x}) - \vartheta_{\mathrm{amb}}\big)\Big\}
}
\, \mathrm{d} \Omega \notag \\
&\qquad - \frac{1}{f_0} \int_{\Sigma} \big\llangle \vartheta^{\prime}(\mathbf{x};\kappa(\mathbf{x}))\big\rrangle \, 
{\color{purple}
\Big\{
\big\llbracket d \, \kappa(\mathbf{x}) \mathrm{grad}[\lambda(\mathbf{x})] \big\rrbracket 
- \chi \, \big\llangle \mathrm{grad}[\lambda(\mathbf{x})] \bullet
\widehat{\mathbf{t}}(\mathbf{x}) \big\rrangle 
\Big\}
}
\, \mathrm{d} \Gamma \notag \\
&\qquad + \frac{1}{f_0} \int_{\Sigma} \big\llangle d \, \kappa(\mathbf{x}) \mathrm{grad}[\vartheta^{\prime}(\mathbf{x};\kappa(\mathbf{x}))]\big\rrangle \bullet 
{\color{purple}
\Big\{\llbracket \lambda(\mathbf{x}) - \lambda_{\mathrm{amb}}\rrbracket \Big\} 
}
\, \mathrm{d} \Gamma \notag \\
&\qquad + \frac{1}{f_0} \int_{\partial \Omega} \vartheta^{\prime}(\mathbf{x};\kappa(\mathbf{x})) \, 
{\color{purple}
\Big\{-d \, 
\kappa(\mathbf{x}) \mathrm{grad}[\lambda(\mathbf{x})] 
\bullet \widehat{\mathbf{n}}(\mathbf{x}) \Big\}
}
\, \mathrm{d} \Gamma \notag \\
&\qquad - \frac{\chi}{f_0}
\big\llangle \vartheta^{\prime}(\mathbf{x};\kappa(\mathbf{x})) \big \rrangle \, 
{\color{purple}
\Big\{\big\llangle\lambda(\mathbf{x}) 
\big\rrangle - \vartheta_{\mathrm{amb}}\Big\} \, \Big\vert_{s = 1}
}
\end{align}

The corresponding adjoint state problem is obtained by forcing all the terms in parenthesis of Eq.~\eqref{Eqn:Sentivities_Expression_collected_terms} to be zero: 
\begin{subequations}
\begin{alignat}{2}
    \label{Eqn:ROM_Sensitivity_BoE_AP_kappa} 
    -&d \, \mathrm{div}[\kappa(\mathbf{x})\mathrm{grad}[\lambda(\mathbf{x})]] = f_{0}
    - h_{T} \big(\lambda(\mathbf{x}) - \vartheta_{\mathrm{amb}}\big) 
    && \quad \mathrm{in} \; \Omega \\
    \label{Eqn:ROM_Sensitivity_q_jump_condition_AP_kappa} 
    -&d \, \llbracket\kappa(\mathbf{x}) \mathrm{grad}[\lambda(\mathbf{x})]\rrbracket = -\chi \,  \big\llangle\mathrm{grad}[\lambda(\mathbf{x})] \bullet \widehat{\mathbf{t}}(\mathbf{x}) 
    \big \rrangle 
    && \quad \mathrm{on} \; \Sigma \\
    \label{Eqn:ROM_Sensitivity_temp_jump_condition_AP_kappa} 
    &\llbracket\lambda(\mathbf{x}) \rrbracket = 0  
    && \quad \mathrm{on} \; \Sigma \\
    \label{Eqn:ROM_Sensitivity_q_BC_AP_kappa} 
    -&d \, \widehat{\mathbf{n}}(\mathbf{x}) \bullet \kappa(\mathbf{x})\mathrm{grad}[\lambda(\mathbf{x})] = 0 
    && \quad \mathrm{on} \; \partial \Omega \\
    \label{Eqn:ROM_Sensitivity_inlet_BC_AP_kappa} 
    &\big\llangle\lambda(\mathbf{x})\big\rrangle = \vartheta_{\mathrm{amb}} 
    && \quad \mathrm{at} \, s = 1 \, \mathrm{on} \, \Sigma 
\end{alignat}
\end{subequations}
If $\lambda(\mathbf{x})$ satisfies the adjoint problem, the sensitivity \eqref{Eqn:Sentivities_Expression_collected_terms} takes the following compact form:
\begin{align}
    \label{Eqn:ROM_Sensitivity_wrt_conductivity}
    D\Phi[\kappa(\mathbf{x})] 
    &= - \frac{1}{f_0} \int_{\Omega} d \, \mathrm{grad}[\lambda(\mathbf{x})] 
    \bullet \mathrm{grad}[\vartheta(\mathbf{x};\kappa(\mathbf{x}))] \, 
    \mathrm{d} \Omega 
\end{align}
The solution to the adjoint variable is again 
the temperature field under the reverse flow conditions: 
\begin{align}
    \lambda(\mathbf{x}) = \vartheta^{(r)}(\mathbf{x}) 
\end{align}
Thus, the sensitivity of the mean surface temperature to the host's thermal conductivity is:
\begin{align}
    \label{Eqn:Sensitivity_wrt_conductivity_modified}
    D\Phi[\kappa(\mathbf{x})] 
    &= -\frac{1}{f_0}\int_{\Omega} d \, \mathrm{grad}\left[\vartheta^{(r)}(\mathbf{x})\right] 
    \bullet \mathrm{grad}\left[\vartheta^{(f)}(\mathbf{x})\right] \, 
    \mathrm{d} \Omega  
\end{align}
It is instructive to write the above equation in terms of the heat flux vector: 
\begin{align}
    \label{Eqn:Sensitivity_wrt_conductivity_heat_flux_vector}
    D\Phi[\kappa(\mathbf{x})] 
    &= -\frac{1}{f_0} \int_{\Omega} 
    \frac{d}{\kappa^{2}(\mathbf{x})} \, 
    \mathbf{q}^{(r)}(\mathbf{x}) \bullet 
    \mathbf{q}^{(f)}(\mathbf{x}) \;
    \mathrm{d} \Omega  
\end{align}
where the heat flux vector fields under the forward and reverse flow conditions take the following form: 
\begin{align}
    \mathbf{q}^{(f)}(\mathbf{x}) = - \kappa(\mathbf{x}) 
    \, \mathrm{grad}[\vartheta^{(f)}(\mathbf{x})]
    \; \mathrm{and} \; 
    \mathbf{q}^{(r)}(\mathbf{x}) = - \kappa(\mathbf{x}) 
    \, \mathrm{grad}[\vartheta^{(r)}(\mathbf{x})]
\end{align}
Expression \eqref{Eqn:Sensitivity_wrt_conductivity_heat_flux_vector} suggests that $D\Phi[\kappa(\mathbf{x})]$ can be positive or negative depending on whether $\mathbf{q}^{(f)}(\mathbf{x})$ opposes $\mathbf{q}^{(r)}(\mathbf{x})$, at least in a significant portion of the domain. This observation further indicates that the trend---variation of the sensitivity with thermal conductivity---might not be monotonic. Needless to say, the exact expressions for $\vartheta^{(f)}(\mathbf{x})$ and $\vartheta^{(r)}(\mathbf{x})$---hence $\mathbf{q}^{(f)}(\mathbf{x})$ and $\mathbf{q}^{(r)}(\mathbf{x})$---are not known \emph{a priori}. We, therefore, resort to numerics for establishing the remarked trend.

\section{NUMERICAL VERIFICATION}
\label{Sec:S6_Sensitivity_NR}
We now verify numerically the theoretical findings presented in the previous two sections. All numerical results are generated by implementing the single-field Galerkin formulation using the weak form capability in \citet{COMSOL}. The Galerkin formulation corresponding to the boundary value problem \eqref{Eqn:Sensitivity_BoE}--\eqref{Eqn:Sensitivity_inlet_BC} reads: Find $\vartheta(\mathbf{x}) \in \mathcal{U}$ such that we have 
\begin{align}
    \int_{\Omega} d \; \mathrm{grad}[\delta\vartheta(\mathbf{x})] 
    \bullet \kappa(\mathbf{x}) \,  \mathrm{grad}[\vartheta(\mathbf{x})] 
    \; \mathrm{d} \Omega 
    + \int_{\Omega} \delta \vartheta(\mathbf{x}) 
    \, h_T \, (\vartheta(\mathbf{x}) - \vartheta_{\mathrm{amb}}) 
    \; \mathrm{d} \Omega \notag \\
    + \int_{\Sigma} \delta \vartheta(\mathbf{x}) 
    \, \chi \, 
    \mathrm{grad}[\vartheta(\mathbf{x})] \bullet 
    \widehat{\mathbf{t}}(\mathbf{x}) \;  
    \mathrm{d} \Gamma 
    = \int_{\Omega} \delta \vartheta(\mathbf{x}) 
    \, f(\mathbf{x}) 
    \; \mathrm{d} \Omega 
    \quad \forall \delta\vartheta(\mathbf{x}) 
    \in \mathcal{W} 
\end{align}
where the function spaces are defined as follows: 
\begin{subequations}
\begin{align}
\mathcal{U} &:= \left\{\vartheta(\mathbf{x}) \in H^{1}(\Omega) \; \vert \; \vartheta(\mathbf{x}) = \vartheta_{\mathrm{inlet}} = \vartheta_{\mathrm{amb}} 
\; \mathrm{at} \; s = 0 \; \mathrm{on} \; \Sigma \right\} \\
\mathcal{W} &:= \left\{\delta \vartheta(\mathbf{x}) \in H^{1}(\Omega) \; \vert \; \delta \vartheta(\mathbf{x}) = 0 \; \mathrm{at} \; s = 0 \; \mathrm{on} \; \Sigma \right\} 
\end{align}
\end{subequations}
In the above definitions, $H^{1}(\Omega)$ denotes the standard Sobolev space comprising all functions defined over $\Omega$ that are square-integrable alongside their derivatives \citep{ciarlet1978finite}. We have used quadratic Lagrange triangular elements in all the numerical simulations reported in this paper. The chosen meshes were fine enough to resolve steep gradients across the vasculature.

Table \ref{Table:Sensitivity_Simulation_parameters} lists the simulation parameters. The values we have chosen for the dimensions and input parameters (e.g., ambient temperature, volumetric flow rate) are common and reported in several experimental active-cooling studies \citep{devi2021microvascular,pejman2019gradient}. Also, we have shown results spanning three host material systems---glass fiber-reinforced plastic (GFRP) composite, carbon fiber-reinforced plastic (CFRP) composite, and Inconel (an additive manufacturing nickel-based alloy). These materials are popular among microvascular active-cooling applications. 

\subsection{U-shaped vasculature}
Guided by Eq.~\eqref{Eqn:Sensitivity_wrt_conductivity_heat_flux_vector}, we devise a problem that shows  $D\Phi[\kappa]$ can be positive or negative. The train of thought is: we choose a vasculature comprising two parallel segments. If the spacing between them is small, the heat transfers from one segment to the other—the flowing fluid in a part of the vasculature loses net heat. In such a case, swapping the inlet and outlet will make the heat flux vector under the reverse flow conditions oppose that under the forward flow conditions, making $D\Phi[\kappa]$ positive. A concomitant manifestation will be a non-monotonic temperature variation along the vasculature.

\textbf{Figure \ref{Fig:Sensitivity_Ushaped_vasculature}} provides a pictorial description of one such problem. The corresponding heat flux vectors are shown in \textbf{Fig.~\ref{Fig:Sensitivity_Ushaped_Heat_flux_vector}}. Clearly, under the forward flow conditions with close spacing ($l = 20 \; \mathrm{mm}$), the heat flows from the right vertical segment, connected to the outlet, to the left vertical segment (which is connected to the inlet). Further, the heat flux vectors under the two flow conditions oppose each other in the region between these two parallel segments, and the magnitude of the heat flux vector in this sandwiched region is large. Henceforward, we refer to the heat transfer from one segment to another along the vasculature as \emph{countercurrent heat exchange}.

Depending on the strength of the countercurrent heat exchange, increasing the host material's conductivity need not result in a monotonic variation of either the mean surface temperature or thermal efficiency, as conveyed in  \textbf{Fig.~\ref{Fig:Sensitivity_Ushaped_Effect_of_k}}. On the other hand,  \textbf{Fig.~\ref{Fig:Sensitivity_Ushaped_Effect_of_Q}} shows that increasing the flow rate invariably decreases the mean surface temperature, despite countercurrent heat exchanges---verifying once more Theorem \ref{Thm:Sensitivity_DPhiChi}. Nonetheless, what factors---material, geometric and input parameters---promote or hinder countercurrent heat exchange is yet to be studied: worthy of a separate investigation. 

\begin{figure}[h]
    \centering
    \includegraphics[scale=0.6]{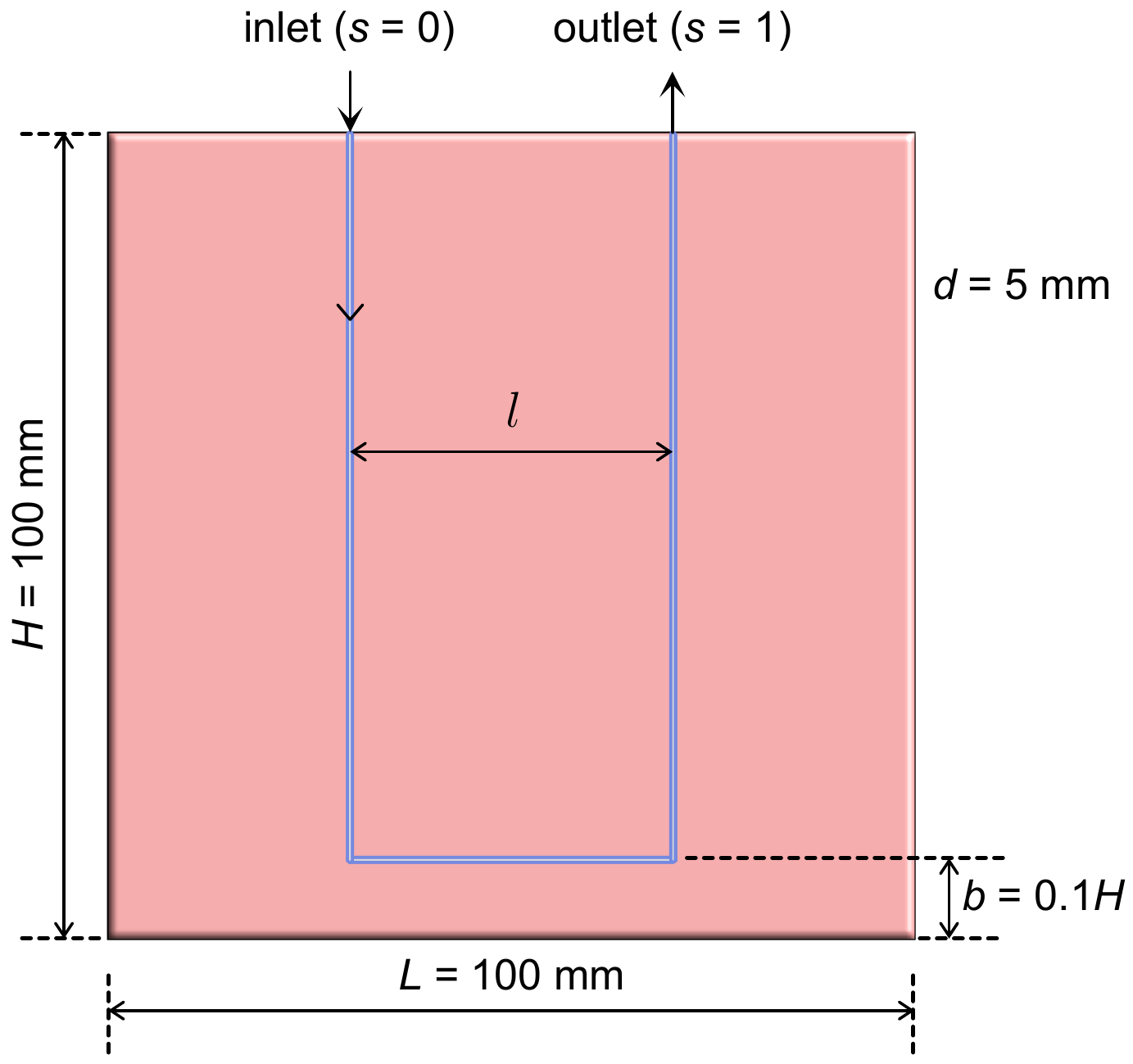}
    \caption{\textsf{U-shaped vasculature: Problem description.} The domain is a square $(100 \, \mathrm{mm} \times 100 \, \mathrm{mm})$ with thickness $d = 5 \, \mathrm{mm}$ and contains a U-shaped vasculature. Fluid flows through the vasculature. The inlet and outlet locations are indicated for forward flow conditions; these locations are swapped for reverse flow. $\vartheta_{\mathrm{inlet}} = \vartheta_{\mathrm{amb}}$ is prescribed at the inlet. A heat source supplies a uniform power: $f(\mathbf{x}) = f_0$, and the entire boundary is adiabatic. $l$ denotes the spacing between the two vertical segments. For small values of $l$, countercurrent heat exchange---heat transfer from one vertical segment to the other---can be dominant.
    \label{Fig:Sensitivity_Ushaped_vasculature}}
\end{figure}

\begin{figure}[htbp]
    \centering
    \includegraphics[scale=0.85]{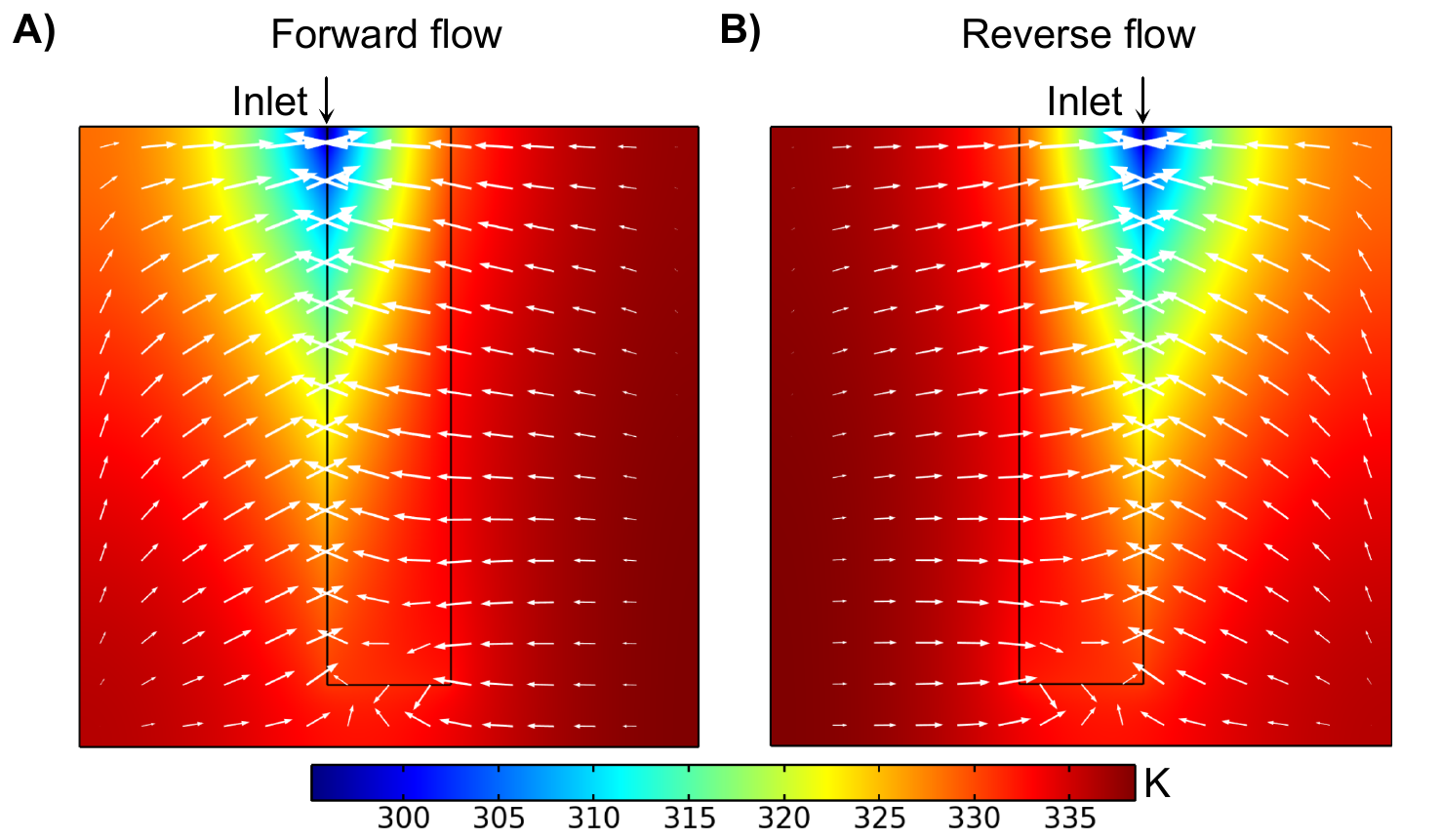}
    \caption{\textsf{U-shaped vasculature: Heat flux vector.} The temperature field superposed with heat flux vectors (white arrows) are shown under forward and reverse flow conditions for $l = 20 \; \mathrm{mm}$. The results are for CFRP with $Q = 1 \, \mathrm{mL/min}$, pumped at the inlet (indicted in the plots). \emph{Observation:} Under both the flow conditions, the heat flux vectors are large in magnitude in the region sandwiched between the vertical segments of the vasculature and, notably, oppose each other.}
    \label{Fig:Sensitivity_Ushaped_Heat_flux_vector}
\end{figure}

\begin{figure}[h]
    \centering
    \includegraphics[scale=0.9]{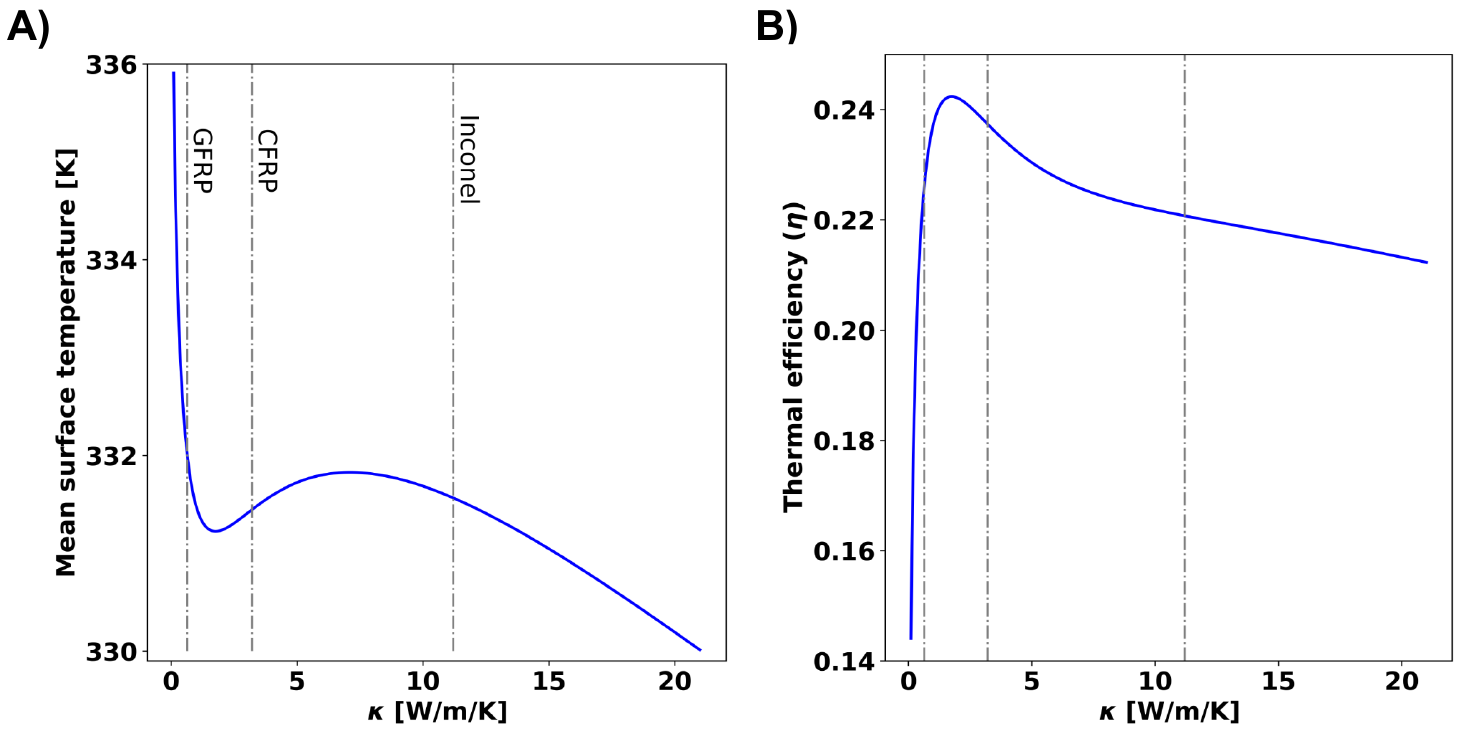}

    \caption{\textsf{U-shaped vasculature: Sensitivity to thermal conductivity.} The variation of the mean surface temperature (A) and thermal efficiency (B) with conductivity are shown for CFRP with $Q = 1 \; \mathrm{mL/min}$. \emph{Inference:}  Increasing thermal conductivity need not result in a monotonic change in MST and thermal efficiency. Depending on the presence and strength of countercurrent heat exchange (see Fig.~\ref{Fig:Sensitivity_Ushaped_Heat_flux_vector}), the heat flux vectors under forward and reverse flow conditions can oppose or align. Consequently, the sensitivity of MST to conductivity, given by expression \eqref{Eqn:Sensitivity_DPhi_chi}, can be either positive or negative.}
    \label{Fig:Sensitivity_Ushaped_Effect_of_k}
\end{figure}

\begin{figure}[h]
    \centering
    \includegraphics[scale=1]{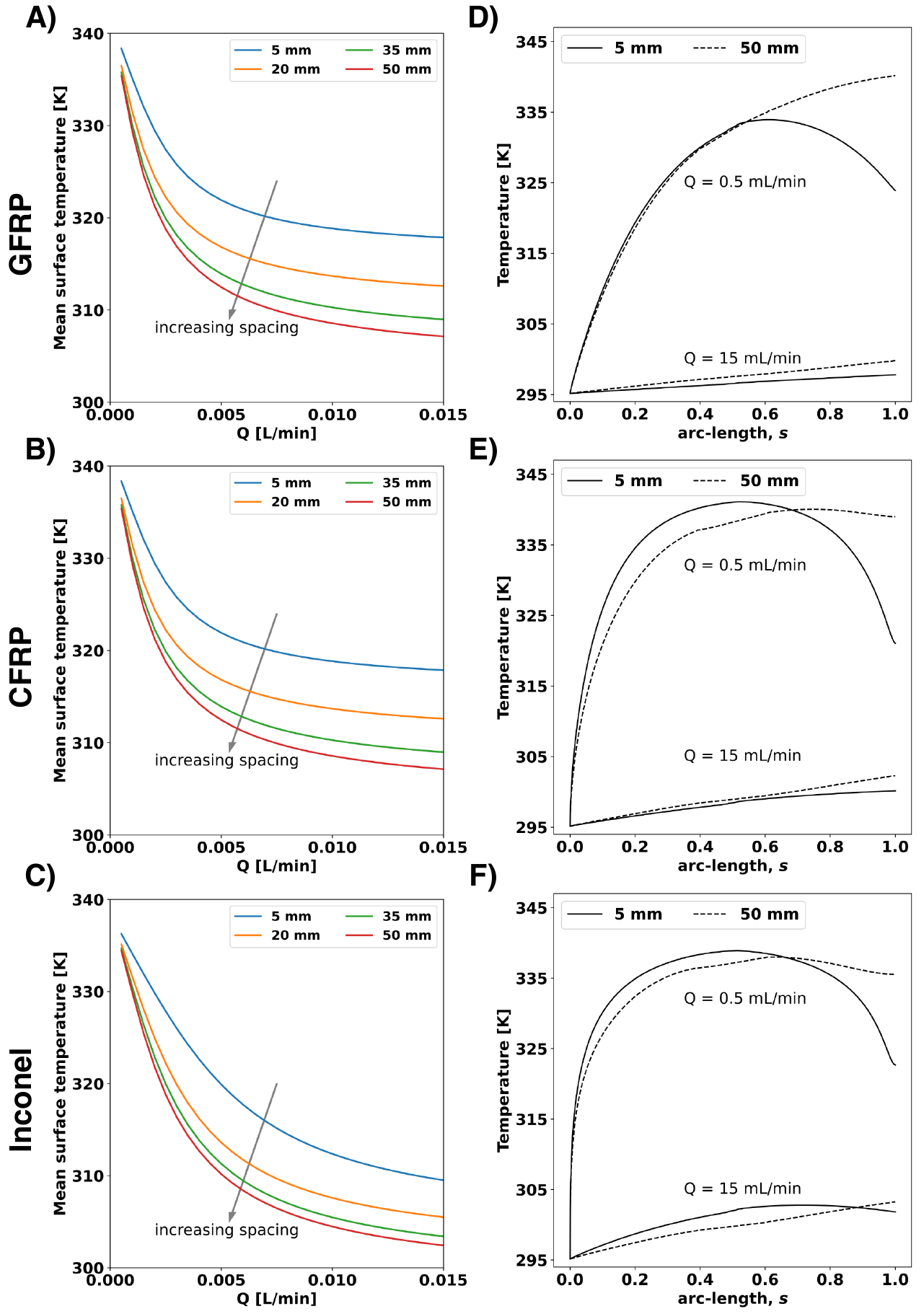}
    \caption{\textsf{U-shaped vasculature: Sensitivity to volumetric flow rate $(Q)$.} The plots on the left panel show the variation of the mean surface temperature with $Q$ for three material systems (GFRP, CFRP, Inconel) and for various spacings $l$. The plots on the right panel report the corresponding temperature variation along the vasculature for different values of $Q$ and $l$. \emph{Inference:} The MST decreases monotonically as $Q$ increases, despite a non-monotonic temperature variation along the vasculature (cf. $Q = 0.5 \; \mathrm{mL/min}$ and $l = 5 \; \mathrm{mm}$). These results verify Theorem \ref{Thm:Sensitivity_DPhiChi}.}
    \label{Fig:Sensitivity_Ushaped_Effect_of_Q}
\end{figure}

\begin{table}[h]
\caption{This table lists the parameters used in numerical simulations. Three host solids are studied: glass fiber-reinforced plastic (GFRP) composite, carbon fiber-reinforced plastic (CFRP) composite, and Inconel 718---an additive manufacturing metal. The fluid is distilled water. The material properties are taken from the literature \citep{devi2021microvascular}. \label{Table:Sensitivity_Simulation_parameters}}
\begin{tabular}{ll}\hline 
Length of the domain $L$ & 100 mm \\
Height of the domain $H$ & 100 mm \\
Thickness $d$ & 5 mm \\
Host solid's thermal conductivity $\kappa$ & 
$\left\{\begin{array}{l} 
0.6360 \; \mathrm{W/m/K} \; \mathrm{(GFRP)} \\
3.2110 \; \mathrm{W/m/K} \; \mathrm{(CFRP)} \\
11.2 \; \mathrm{W/m/K} \; \mathrm{(Inconel, In718)} \\ 
\mbox{or parametric sweep} \\
\end{array} \right.$
\\
Applied heater power $f_0$ & 1000 $\mathrm{W/m^2}$ \\ 
Heat transfer coefficient $h_T$ & 21 $\mathrm{W/m^2/K}$ \\ 
Ambient temperature $\vartheta_{\mathrm{amb}}$ & 295.15 K 
(22 $^\circ$C) \\
Inlet temperature $\vartheta_{\mathrm{Inlet}}$ & $\vartheta_{\mathrm{amb}}$ \\ 
Specific heat capacity of the fluid $c_f$ & 4183 $\mathrm{J/kg/K}$ \\
Density of the fluid $\rho_f$ & 1000 $\mathrm{kg/m^3}$ \\ 
Volumetric flow rate $Q$ & varies (see individual problem descriptions) \\ 
\hline 
\end{tabular}
\end{table}

\subsection{Straight channel}
To verify further the role of countercurrent heat exchange, we consider a straight channeled vasculature, illustrated in \textbf{Fig.~\ref{Fig:Sensitivity_Straight_vasculature}}. The lack of nearby segments implies there will be no countercurrent heat exchange. The heat flux vectors under the forward and reverse flow conditions align (more or less) in the same direction, as exhibited in \textbf{Fig.~\ref{Fig:Sensitivity_Straight_Heat_flux_vector}}. So, the mean surface temperature should decrease monotonically as the thermal conductivity increases, verified for various flow rates in \textbf{Fig.~\ref{Fig:Sensitivity_Straight_Effect_of_k}}. \textbf{Figure \ref{Fig:Sensitivity_Straight_Effect_of_Q}} verifies Theorem \ref{Thm:Sensitivity_DPhiChi}: Increasing the volumetric flow rate (keeping the fluid fixed) will decrease the mean surface temperature.

\begin{figure}[h]
    \centering
    \includegraphics[scale=0.6]{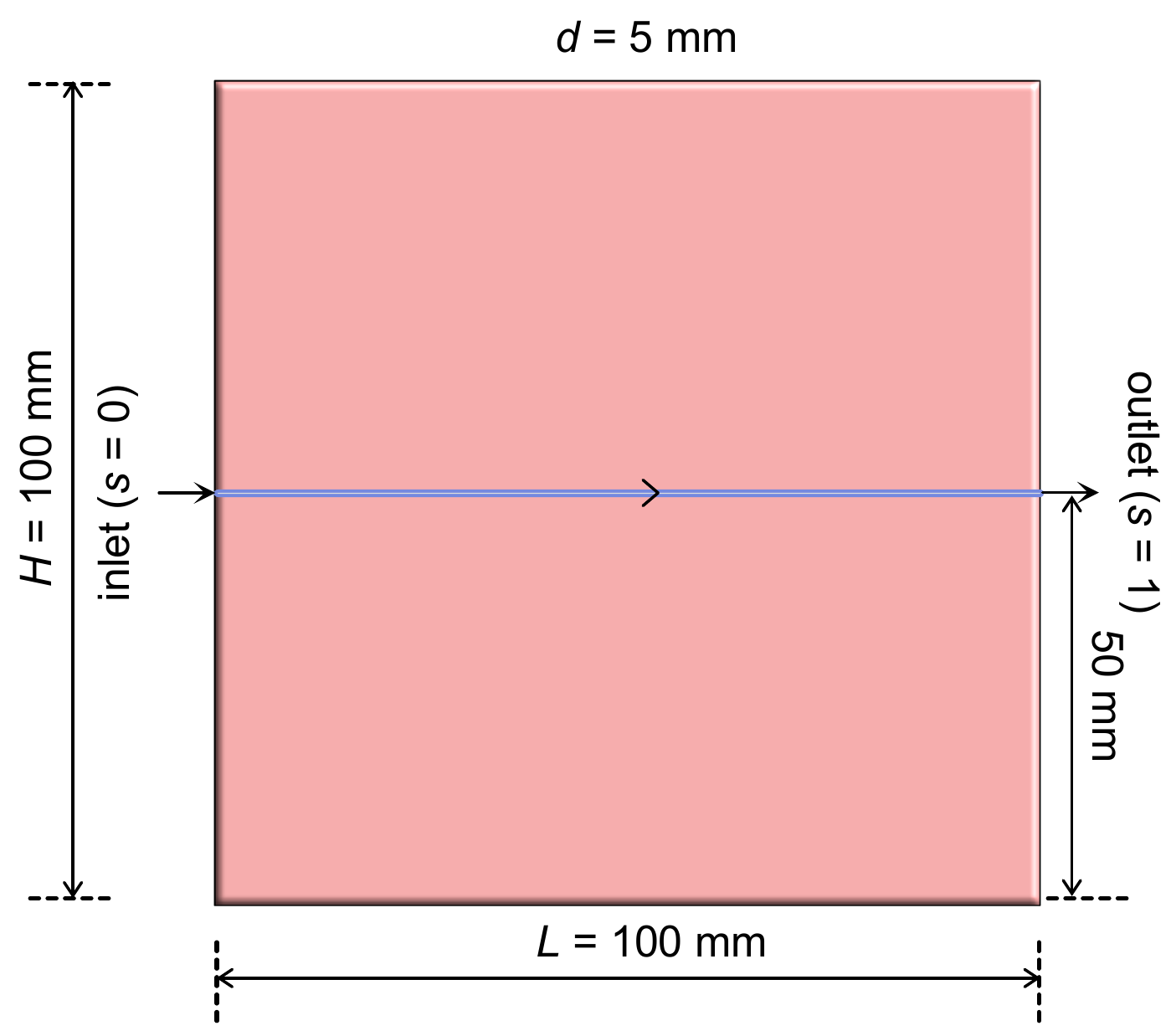}
    \caption{\textsf{Straight channel: Problem description.} The domain is a square $(100 \, \mathrm{mm} \times 100 \, \mathrm{mm})$ with thickness $d = 5 \, \mathrm{mm}$ and contains a straight channel. Fluid flows through the vasculature. The inlet and outlet locations are indicated. $\vartheta_{\mathrm{inlet}} = \vartheta_{\mathrm{amb}}$ is prescribed at the inlet. A heat source supplies a uniform power: $f(\mathbf{x}) = f_0$, and the entire boundary is adiabatic. Because of the absence of parallel segments, there will be no countercurrent heat exchange under this vascular layout. 
    \label{Fig:Sensitivity_Straight_vasculature}}
\end{figure}

\begin{figure}[h]
    \centering
    \includegraphics[scale=0.85]{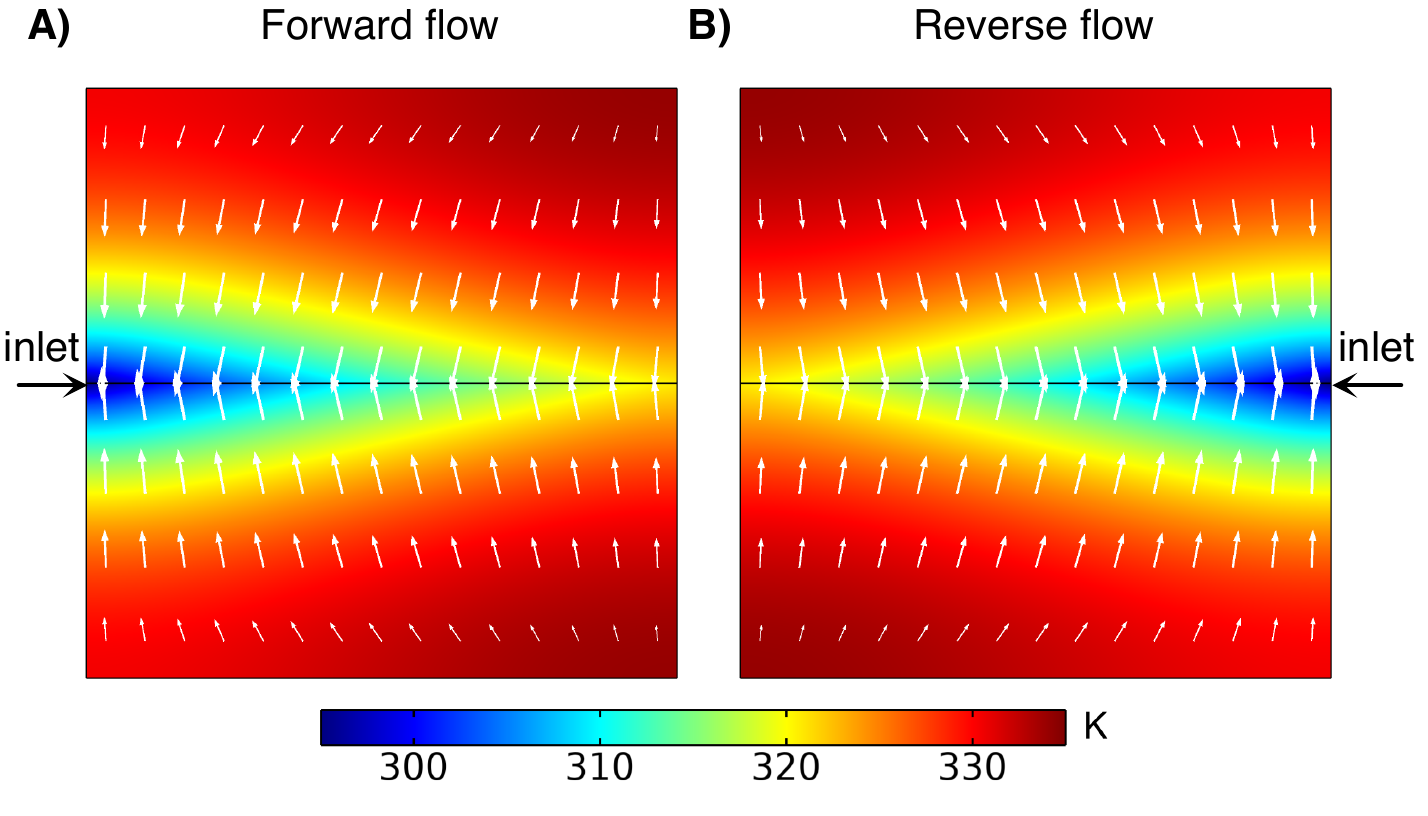}
    \caption{\textsf{Straight channel: Heat flux vector.} The temperature field superposed with heat flux vector (indicated by white arrows) under forward and reverse flow conditions are shown. These results are obtained for CFRP with $Q = 2 \, \mathrm{mL/min}$. \emph{Observation:} The heat flux vectors $\mathbf{q}^{(f)}(\mathbf{x})$ and $\mathbf{q}^{(r)}(\mathbf{x})$, more or less, align with each other. This alignment of heat flux vectors should render $D\Phi[\kappa(\mathbf{x})]$, given by Eq.~\eqref{Eqn:Sensitivity_wrt_conductivity_heat_flux_vector}, to be negative.}
    \label{Fig:Sensitivity_Straight_Heat_flux_vector}
\end{figure}

\begin{figure}
    \centering
    \includegraphics[scale=0.4]{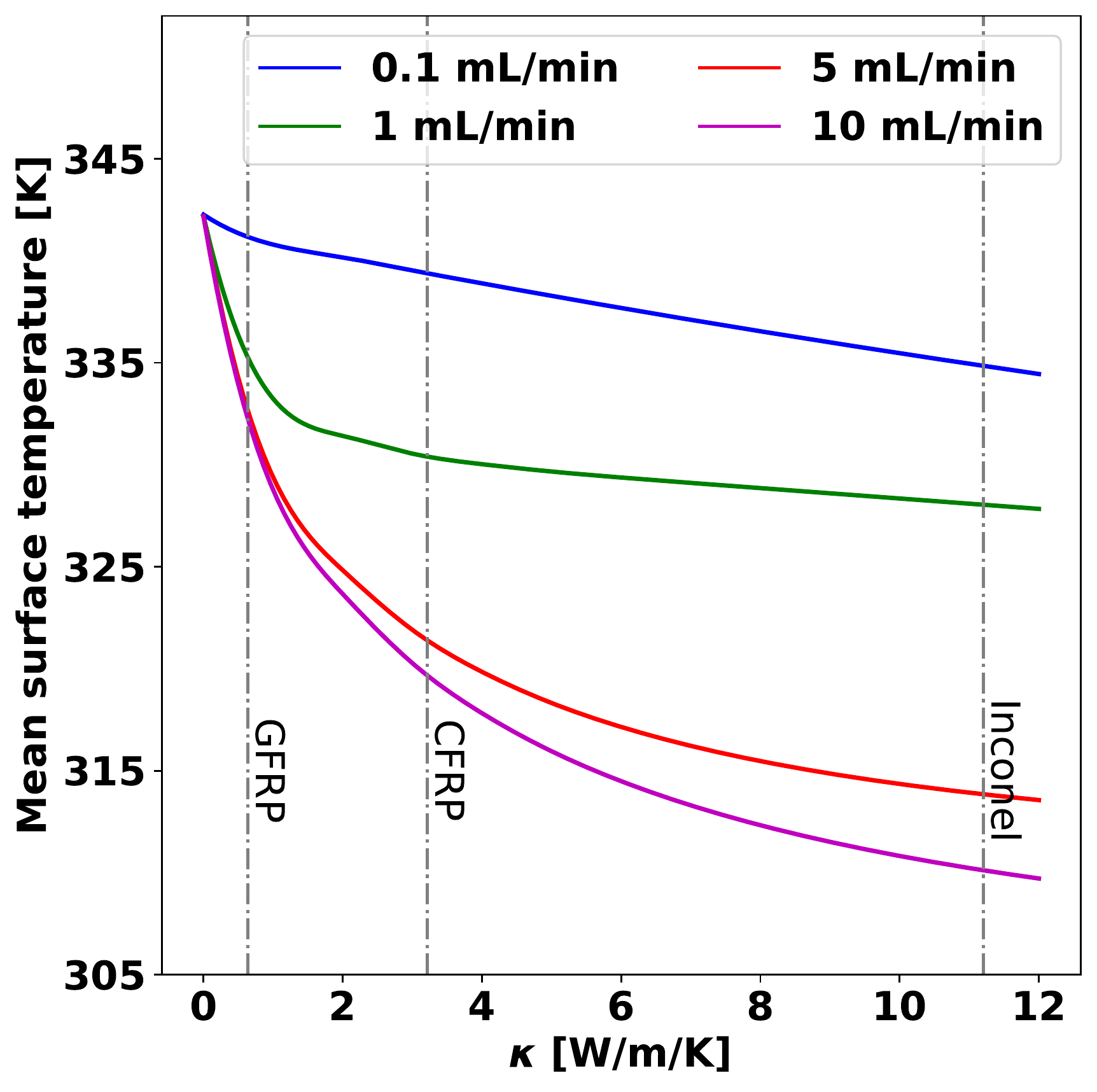}
    \caption{\textsf{Straight channel: Sensitivity to thermal conductivity.} The mean surface temperature decreases monotonically with an increase in conductivity, as the heat flux vectors under the forward and reverse flow conditions more or less align in most of the domain (see Fig.~\ref{Fig:Sensitivity_Straight_Heat_flux_vector}).}
    \label{Fig:Sensitivity_Straight_Effect_of_k}
\end{figure}

\begin{figure}
    \centering
    \includegraphics[scale=1]{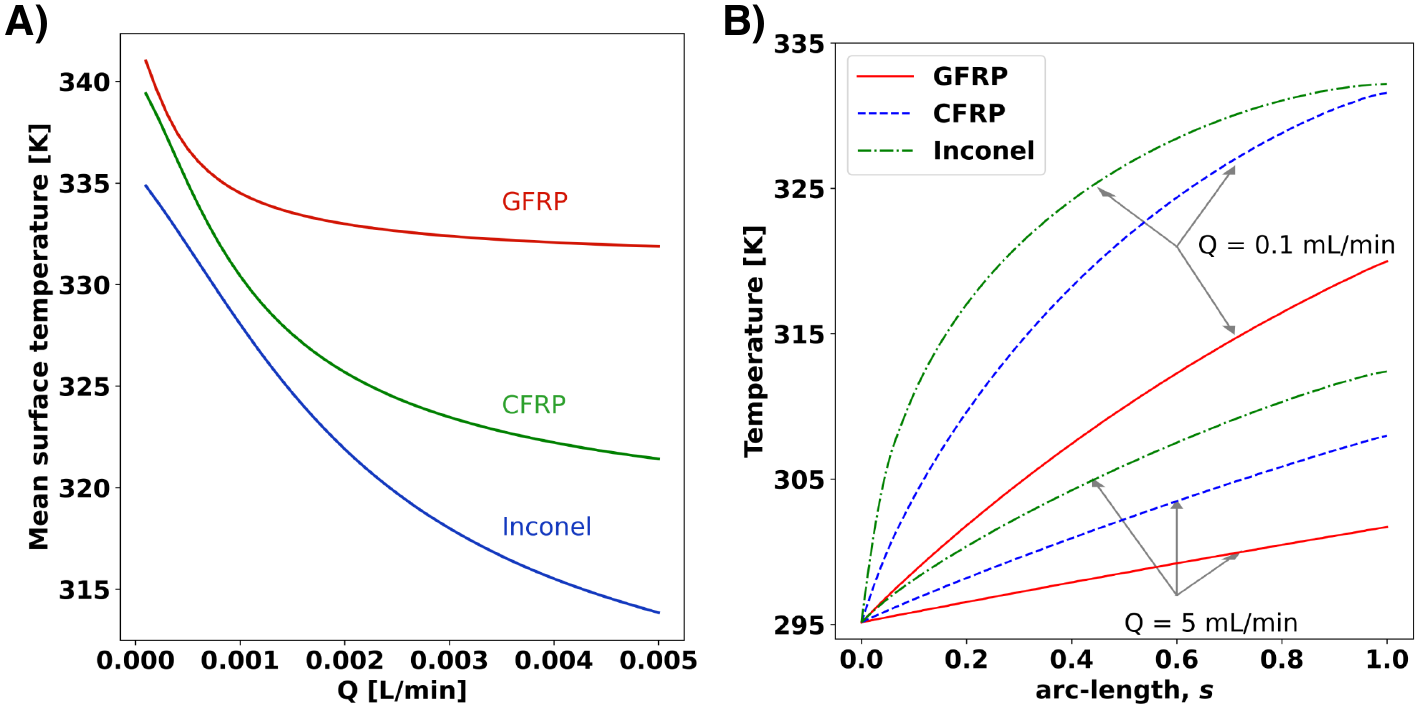}
    \caption{\textsf{Straight channel: Sensitivity to volumetric flow rate.} 
    A) The mean surface temperature decreases monotonically with increasing volumetric flow rate. B) The temperature increases monotonically along the vasculature, indicating that the heat transfers from the host material to the flowing fluid throughout the straight channel. \emph{Inference:} These results are consistent with Theorem \ref{Thm:Sensitivity_DPhiChi}.
    }
    \label{Fig:Sensitivity_Straight_Effect_of_Q}
\end{figure}

\subsection{Serpentine vasculature}
\textbf{Figure \ref{Fig:Sensitivity_Serpentine_vasculature}} depicts the serpentine vasculature layout. Several studies have used this layout (e.g., \citep{devi2021microvascular}) because of its spatial spread, enabling cooling over the entire domain. Clearly, heat transfer under this layout is intricate along the vasculature because of many nearby segments, as portrayed in \textbf{Fig.~\ref{Fig:Sensitivity_Serpentine_Heat_flux_vector}}. On many portions of the vasculature, heat transfers from the host material to flowing fluid on one side and the opposite on the juxtaposed side. Due to this prominent countercurrent heat exchanges, the effect of the thermal conductivity on the mean surface temperature can be multifaceted and depends on the flow rate, as shown in  \textbf{Fig.~\ref{Fig:Sensitivity_Serpentine_Effect_of_k}}. Despite the complex heat transfer map, \textbf{Fig.~\ref{Fig:Sensitivity_Serpentine_Effect_of_Q}} shows that the sensitivity of the mean surface temperature to the volumetric flow rate (for a fixed fluid) is always negative---in agreement with Theorem \ref{Thm:Sensitivity_DPhiChi}.

\begin{figure}[h]
    \centering
    \includegraphics[scale=0.6]{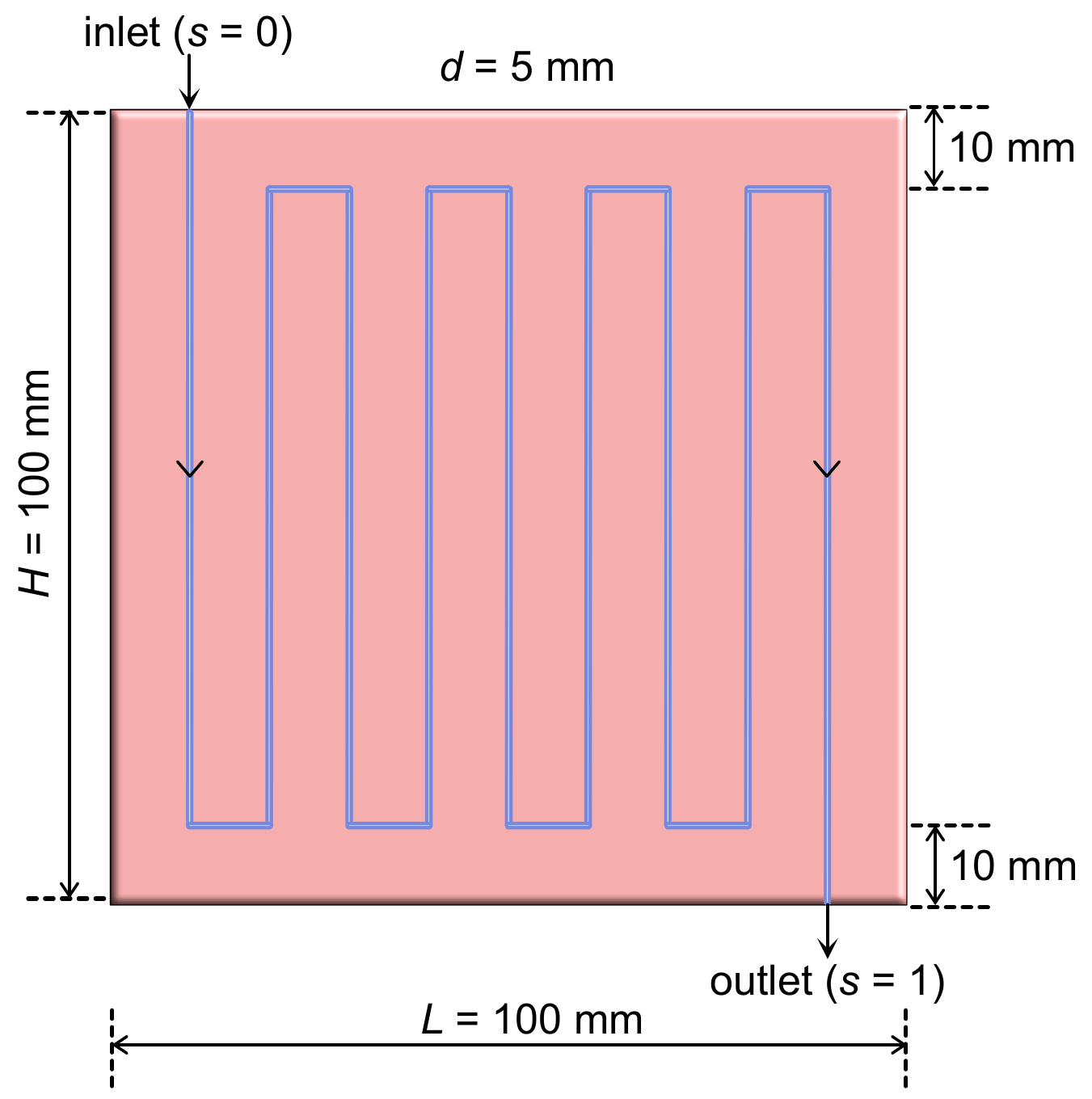}
    \caption{This figure shows a serpentine vasculature with a heat source supplying uniform power to the entire domain. Because of many close parallel segments, there could be notable countercurrent heat exchange across various sections of the vasculature amid active cooling. \label{Fig:Sensitivity_Serpentine_vasculature}}
\end{figure}

\begin{figure}[h]
    \centering
    \includegraphics[scale=0.9]{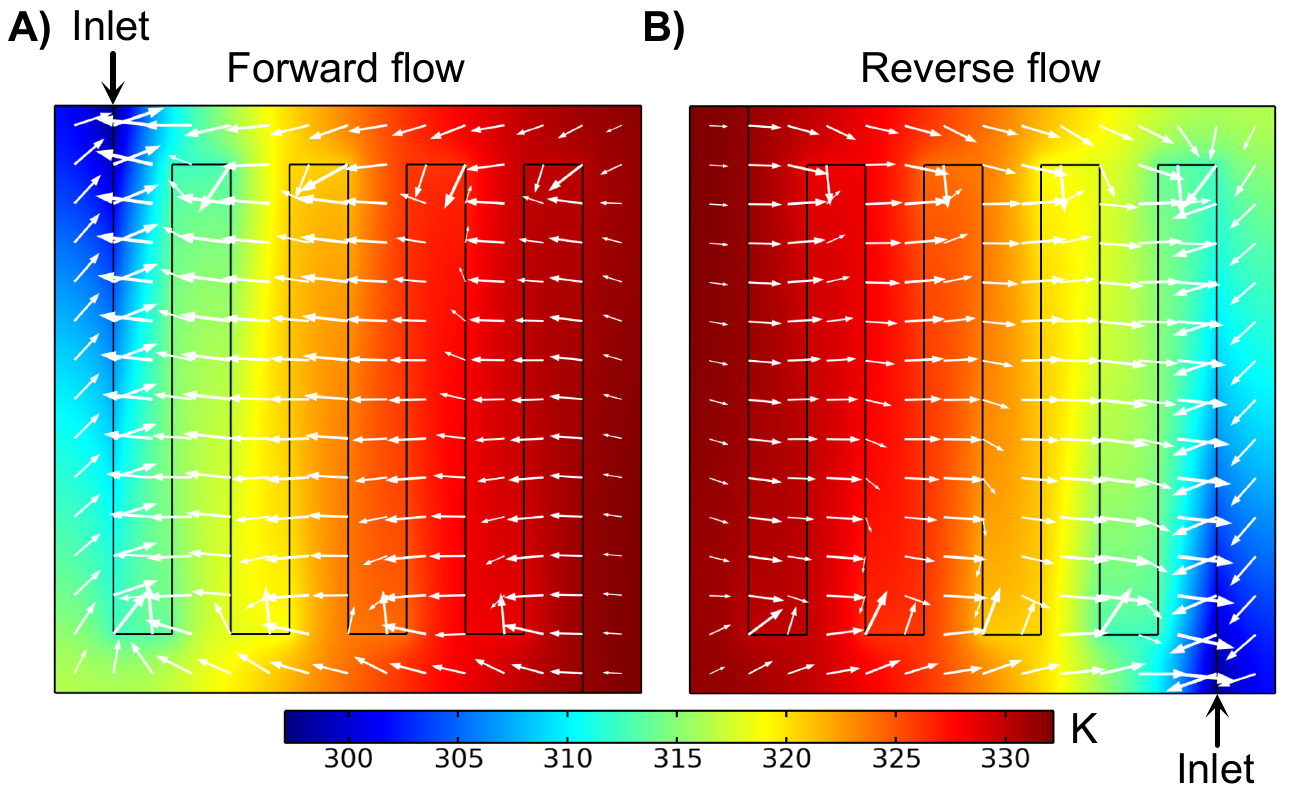}
    \caption{\textsf{Serpentine vasculature: Heat flux vector.} Countercurrent heat exchange is prevalent in the entire domain, as indicated by the (white) arrows denoting the heat flux vector field. The plots in this figure are for CFRP with $Q = 2 \, \mathrm{mL/min}$. Clearly, for this case, the arrows under the forward flow conditions oppose those under the reverse flow conditions, thereby making the sensitivity expression \eqref{Eqn:Sensitivity_wrt_conductivity_modified} positive.}
    \label{Fig:Sensitivity_Serpentine_Heat_flux_vector}
\end{figure}

\begin{figure}[h]
    \centering
    \includegraphics[scale=1.1]{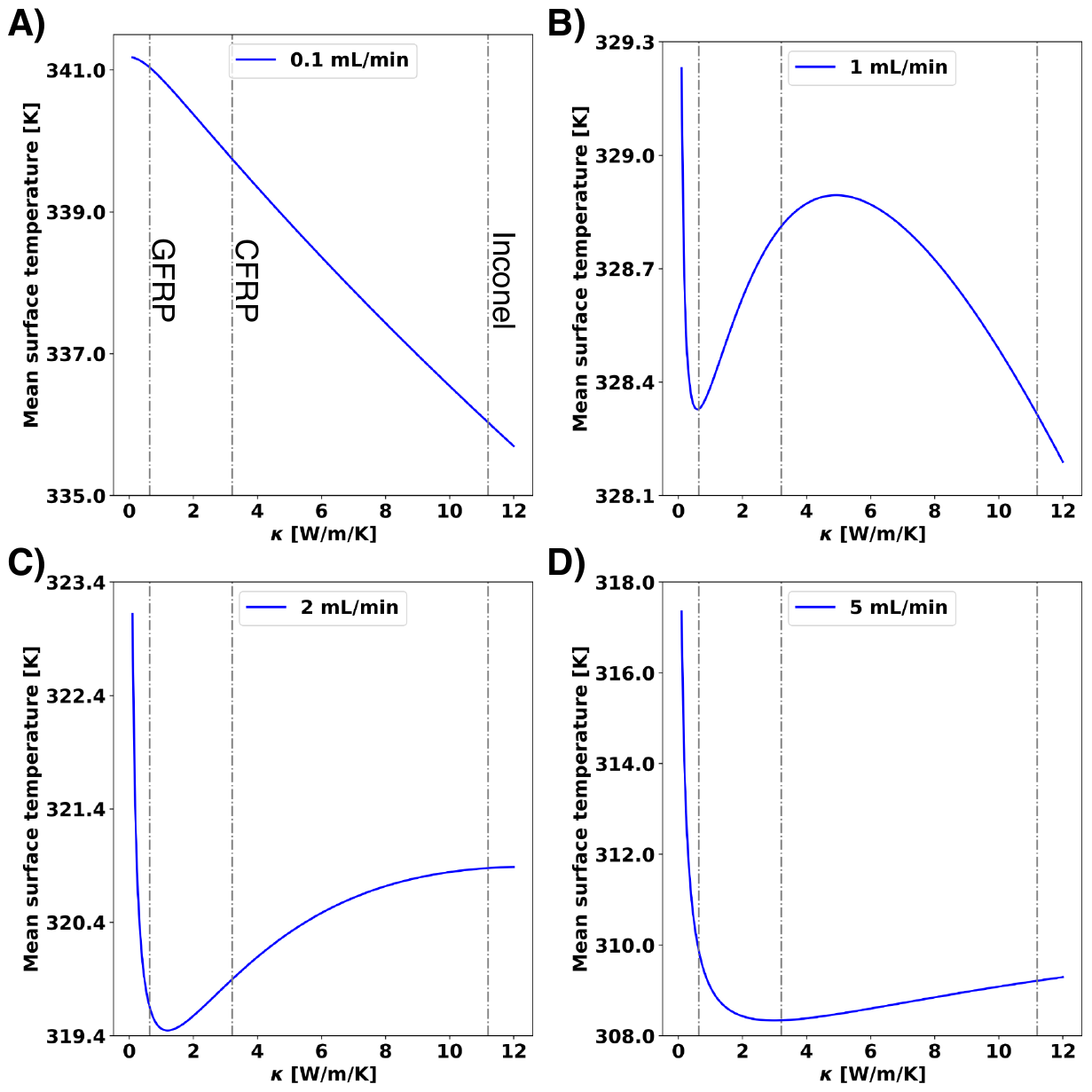}
    \caption{\textsf{Serpentine vasculature: Sensitivity to thermal conductivity}. The variations for various flow rates of the mean surface temperature to the changes in thermal conductivity are shown. Due to the countercurrent heat exchanges, the said variation is not always monotonic but is multifaceted within the range of conductivities often used in active-cooling studies.  \label{Fig:Sensitivity_Serpentine_Effect_of_k}}
\end{figure}

\begin{figure}[h]
    \centering
    \includegraphics[scale=1]{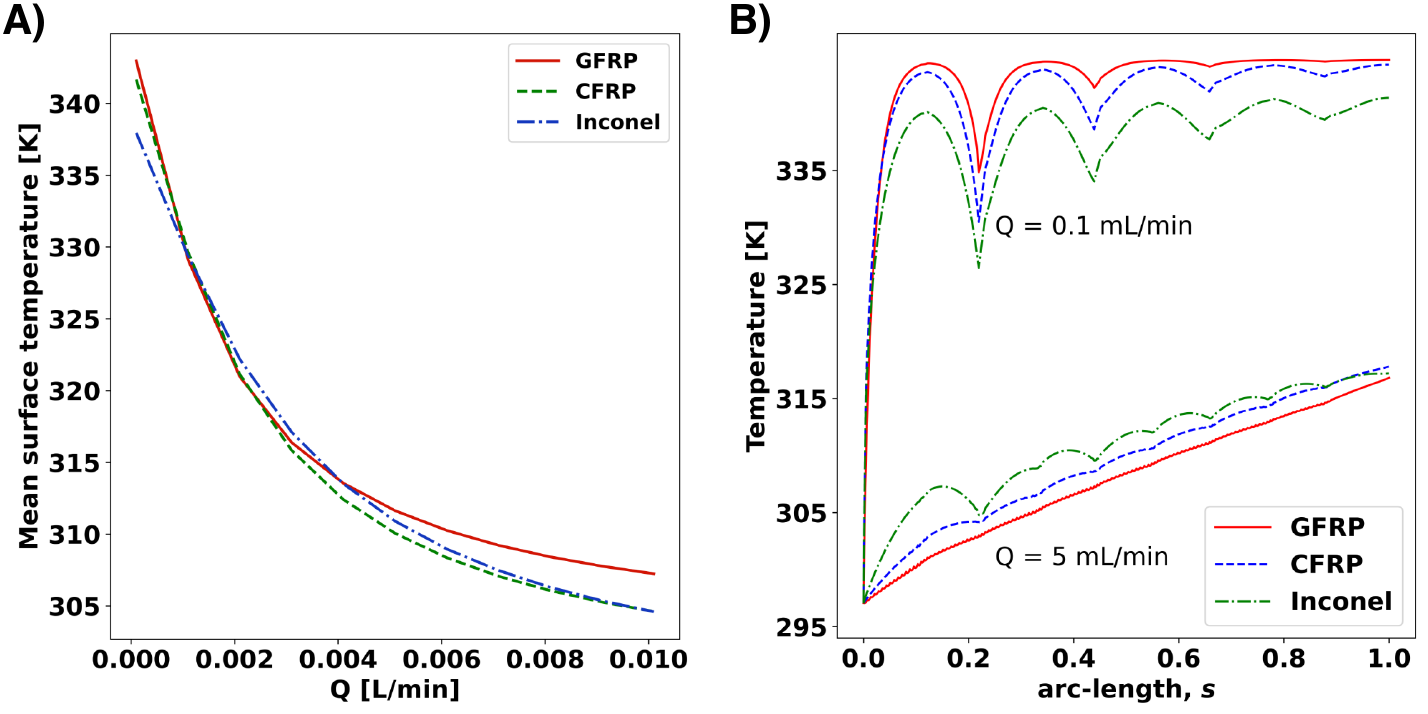}
    \caption{\textsf{Serpentine vasculature: Sensitivity to volumetric flow rate.} A) The variation of $\vartheta_{\mathrm{MST}}$ to the changes in the volumetric flow rate $(Q)$ is shown for various host materials. B) The temperature along the vasculature is plotted, and the variation is non-monotonic, meaning significant countercurrent heat exchanges occur along the vasculature's length. \emph{Inference:} Despite potent countercurrent heat exchange, the mean surface temperature decreases monotonically as $Q$ increases, equivalent to increasing the heat capacity rate for a fixed fluid. These results provide further verification of Theorem \ref{Thm:Sensitivity_DPhiChi}. \label{Fig:Sensitivity_Serpentine_Effect_of_Q}}
\end{figure}

\section{CLOSURE}
\label{Sec:S7_Sensitivity_Closure}

This paper addressed the individual effects of heat capacity rate (product of volumetric flow rate and heat capacity) and thermal conductivity 
on the mean surface temperature (MST) and thermal efficiency. The study avails a reduced-order model for thermal regulation and conducts mathematical analysis based on the adjoint state method (a popular sensitivity analysis approach) and representative numerical simulations. The principal finding on the \emph{sensitivity analysis front} is: the adjoint-state problem is the boundary value problem for the reverse flow conditions (i.e., swapping the inlet and outlet locations). The primary determinations on the \emph{physics front} are: 
\begin{enumerate}[(C1)]
\item Irrespective of the vasculature layout, an increase in the mass flow rate of the circulating fluid decreases MST---meaning that the thermal efficiency increases.
\item However, MST variation is non-monotonic with the host material’s thermal conductivity.
\item The sensitivity of MST to the conductivity is proportional to the weighted inner product of the heat flux vector fields under the forward and reverse flow conditions.
\item When countercurrent heat exchange dominates, these two heat flux vectors oppose each other, thereby making the sensitivity of MST to thermal conductivity positive. The trend can be the opposite if the countercurrent heat exchange is absent or insignificant.
\end{enumerate}

A direct significance of this work is it settles an unresolved fundamental question related to thermal regulation in thin vascular systems: how does the host material's thermal conductivity affect MST? Also, the reported analysis and results (a) enhance our fundamental understanding of vascular-based thermal regulation and (b) provide a clear-cut path to pose material design problems. 

As alluded to in \S\ref{Sec:S6_Sensitivity_NR}, a logical sequel to this study is to address this principal question: what factors (material, geometric and input parameters) promote or hinder countercurrent heat exchange in microvascular active-cooling systems? Further, we envision scientific explorations on two fronts:
\begin{enumerate}
\item An experimental program---validating the identified non-monotonic behavior of the sensitivity of thermal conductivity on the mean surface temperature---will benefit the field. Further, these experiments should realize \emph{countercurrent heat exchange} and confirm its role in the sign change of this sensitivity.
\item On the modeling front, a natural extension is to develop a material design framework and comprehend the resulting designs. Also, researchers should explore alternative objective functionals appropriate to thermal regulation (other than the mean surface temperature) and study the ramifications of such alternatives. The results from this paper provide the necessary impetus to undertake the remarked material design research. 
\end{enumerate}

\appendix

\section{Algebra and calculus of jumps}

\subsection{Algebra}
\label{Subsec:Sensitivity_App_Algebra}
Below we provide proof of identity \eqref{Eqn:Sensitivities_avg_jump_identities_a}. Proof of the other equivalence \eqref{Eqn:Sensitivities_avg_jump_identities_b} follows a similar procedure. 
\begin{proposition}
    \label{Prop:Sensitivity_jump_identity}
    Given a scalar field $\alpha(\mathbf{x})$ and a vector field $\mathbf{a}(\mathbf{x})$, the following identity holds: 
    \begin{align}
        \llbracket 
        \alpha(\mathbf{x}) \, 
        \mathbf{a}(\mathbf{x}) 
        \rrbracket 
        = \llbracket \alpha(\mathbf{x}) \rrbracket 
        \bullet 
        \llangle \mathbf{a}(\mathbf{x}) \rrangle 
        + \llangle \alpha(\mathbf{x}) 
        \rrangle \, \llbracket \mathbf{a}(\mathbf{x}) \rrbracket
    \end{align}
\end{proposition}
\begin{proof}
We use the definitions of the average and jump operators (i.e., Eqs.~\eqref{Eqn:Sensitivity_average_operator} and \eqref{Eqn:Sensitivity_jump_operator}) to expand the first term on the right side of the identity: 
\begin{align}
    \llbracket \alpha(\mathbf{x}) 
    \rrbracket \bullet \llangle 
    \mathbf{a}(\mathbf{x}) \rrangle 
    &= \Big(\alpha^{+}(\mathbf{x}) \,  \widehat{\mathbf{n}}^{+}(\mathbf{x})
    + \alpha^{-}(\mathbf{x}) \,  \widehat{\mathbf{n}}^{-}(\mathbf{x})
    \Big) \bullet 
    \frac{1}{2}
    \Big(\mathbf{a}^{+}(\mathbf{x}) 
    + \mathbf{a}^{-}(\mathbf{x}) \Big) \notag \\
    &= \frac{1}{2} \alpha^{+}(\mathbf{x}) \,  \widehat{\mathbf{n}}^{+}(\mathbf{x})
    \bullet \mathbf{a}^{+}(\mathbf{x})
    + \frac{1}{2} \alpha^{+}(\mathbf{x}) \,  \widehat{\mathbf{n}}^{+}(\mathbf{x})
    \bullet \mathbf{a}^{-}(\mathbf{x}) \notag \\
    &\qquad +\frac{1}{2} \alpha^{-}(\mathbf{x}) \,  \widehat{\mathbf{n}}^{-}(\mathbf{x})
    \bullet \mathbf{a}^{+}(\mathbf{x})
    + \frac{1}{2} \alpha^{-}(\mathbf{x}) \,  \widehat{\mathbf{n}}^{-}(\mathbf{x})
    \bullet \mathbf{a}^{-}(\mathbf{x})
\end{align}
Likewise, expanding the second term on the right side of the identity, we get:
\begin{align}
    \llangle \alpha(\mathbf{x}) 
    \rrangle \, \llbracket
    \mathbf{a}(\mathbf{x}) \rrbracket 
    &= \frac{1}{2}
    \Big(\alpha^{+}(\mathbf{x}) 
    + \alpha^{-}(\mathbf{x}) \Big) \,  \Big(\mathbf{a}^{+}(\mathbf{x}) \bullet  \widehat{\mathbf{n}}^{+}(\mathbf{x})
    + \mathbf{a}^{-}(\mathbf{x}) \bullet   \widehat{\mathbf{n}}^{-}(\mathbf{x})
    \Big) \notag \\
    &= \frac{1}{2} \alpha^{+}(\mathbf{x}) \,  \mathbf{a}^{+}(\mathbf{x})
    \bullet \widehat{\mathbf{n}}^{+}(\mathbf{x})
    + \frac{1}{2} \alpha^{+}(\mathbf{x}) \,  \mathbf{a}^{-}(\mathbf{x})
    \bullet \widehat{\mathbf{n}}^{-}(\mathbf{x}) \notag \\
    &\qquad +\frac{1}{2} \alpha^{-}(\mathbf{x}) \,  \mathbf{a}^{+}(\mathbf{x})
    \bullet \widehat{\mathbf{n}}^{+}(\mathbf{x})
    + \frac{1}{2} \alpha^{-}(\mathbf{x}) \,  \mathbf{a}^{-}(\mathbf{x})
    \bullet \widehat{\mathbf{n}}^{-}(\mathbf{x})
\end{align}
By adding the above two equations and noting the commutative property of the dot product, we get 
\begin{align}
    \llbracket \alpha(\mathbf{x}) 
    \rrbracket \bullet \llangle 
    \mathbf{a}(\mathbf{x}) \rrangle
    + 
    \llangle \alpha(\mathbf{x}) 
    \rrangle \, \llbracket
    \mathbf{a}(\mathbf{x}) \rrbracket 
    &= \alpha^{+}(\mathbf{x}) \,  \mathbf{a}^{+}(\mathbf{x})
    \bullet \widehat{\mathbf{n}}^{+}(\mathbf{x})
    + \alpha^{-}(\mathbf{x}) \,  \mathbf{a}^{-}(\mathbf{x})
    \bullet \widehat{\mathbf{n}}^{-}(\mathbf{x}) \notag \\
    &\qquad + \frac{1}{2} 
    \Big(\alpha^{+}(\mathbf{x}) \,  \mathbf{a}^{-}(\mathbf{x}) 
    + \alpha^{-}(\mathbf{x}) \,  \mathbf{a}^{+}(\mathbf{x})
    \Big) 
    \bullet \Big(\widehat{\mathbf{n}}^{+}(\mathbf{x})  
    + \widehat{\mathbf{n}}^{-}(\mathbf{x}) 
    \Big) 
\end{align}
Noting that $\mathbf{n}^{+}(\mathbf{x}) + \mathbf{n}^{-}(\mathbf{x}) = \mathbf{0}$, we establish the desired result as follows: 
\begin{align}
    \llbracket \alpha(\mathbf{x}) 
    \rrbracket \bullet \llangle 
    \mathbf{a}(\mathbf{x}) \rrangle
    + 
    \llangle \alpha(\mathbf{x}) 
    \rrangle \, \llbracket
    \mathbf{a}(\mathbf{x}) \rrbracket 
    &= \alpha^{+}(\mathbf{x}) \,  \mathbf{a}^{+}(\mathbf{x})
    \bullet \widehat{\mathbf{n}}^{+}(\mathbf{x})
    + \alpha^{-}(\mathbf{x}) \,  \mathbf{a}^{-}(\mathbf{x})
    \bullet \widehat{\mathbf{n}}^{-}(\mathbf{x}) \notag \\
    &= 
    \Big(\alpha(\mathbf{x}) \,  \mathbf{a}(\mathbf{x})\Big)^{+}
    \bullet \widehat{\mathbf{n}}^{+}(\mathbf{x})
    + \Big(\alpha(\mathbf{x}) \,  \mathbf{a}(\mathbf{x})\Big)^{-}
    \bullet \widehat{\mathbf{n}}^{-}(\mathbf{x}) \notag \\
    &= \llbracket \alpha(\mathbf{x}) \, \mathbf{a}(\mathbf{x}) \rrbracket
\end{align}
\end{proof}

\subsection{Calculus}
\label{Subsec:Sensitivity_App_Calculus}
The divergence theorem over $\Omega \setminus \Sigma$ takes the following form: 
\begin{align}
    \label{Eqn:Sensitivity_Divergence_theorem}
    \int_{\Omega \setminus \Sigma} 
    \mathrm{div}[\mathbf{a}(\mathbf{x})] \, \mathrm{d} \Omega 
    = \int_{\partial \Omega} 
    \mathbf{a}(\mathbf{x}) \bullet \widehat{\mathbf{n}}(\mathbf{x}) 
    \, \mathrm{d} \Gamma
    + \int_{\Sigma} \llbracket 
    \mathbf{a}(\mathbf{x}) 
    \rrbracket \, \mathrm{d} \Gamma 
\end{align}
where $\mathbf{a}(\mathbf{x})$ is a vector field. The Green's theorem over $\Omega \setminus \Sigma$ can be written as follows: 
\begin{align}
    \label{Eqn:Sensitivity_Greens_theorem}
    \int_{\Omega \setminus \Sigma} 
    \alpha(\mathbf{x}) \,  \mathrm{div}[\mathbf{a}(\mathbf{x})] \, \mathrm{d} \Omega 
    = \int_{\partial \Omega} 
    \alpha(\mathbf{x}) \, \mathbf{a}(\mathbf{x}) 
    \bullet \widehat{\mathbf{n}}(\mathbf{x}) 
    \, \mathrm{d} \Gamma 
    + \int_{\Sigma} 
    \llbracket 
    \alpha(\mathbf{x}) \, \mathbf{a}(\mathbf{x}) 
    \rrbracket
    \, \mathrm{d} \Gamma \notag \\
    - \int_{\Omega \setminus \Sigma} 
    \mathrm{grad}[\alpha(\mathbf{x})] 
    \bullet \mathbf{a}(\mathbf{x})
    \, \mathrm{d} \Omega  
\end{align}
where $\alpha(\mathbf{x})$ is a scalar field. The above two expressions are valid even if the vasculature $\Sigma$ comprises branches. The following result, an application of Green's theorem \eqref{Eqn:Sensitivity_Greens_theorem}, will be useful in deriving the adjoint state problem.
\begin{proposition}
   \label{Prop:Sensitivity_Weighted_theorem}
   Let $\beta(\mathbf{x})$ and $\gamma(\mathbf{x})$ are two smooth scalar fields over $\Omega \setminus \Sigma$. These fields satisfy 
   {\small
   \begin{align}
        \int_{\Omega \setminus \Sigma} 
        \gamma(\mathbf{x}) \,  \mathrm{div}\big[\kappa(\mathbf{x}) \,  \mathrm{grad}[\beta(\mathbf{x})]\big] \, \mathrm{d} \Omega 
        &= \int_{\partial \Omega} 
        \gamma(\mathbf{x}) \, \kappa(\mathbf{x}) \, \mathrm{grad}[\beta(\mathbf{x})]
        \bullet \widehat{\mathbf{n}}(\mathbf{x}) 
        \, \mathrm{d} \Gamma 
        -\int_{\partial \Omega} 
        \beta(\mathbf{x}) \, \kappa(\mathbf{x}) \, \mathrm{grad}[\gamma(\mathbf{x})]
        \bullet \widehat{\mathbf{n}}(\mathbf{x}) 
        \, \mathrm{d} \Gamma 
        \notag \\
        &\qquad + \int_{\Sigma} 
        \big\llbracket 
        \gamma(\mathbf{x}) 
         \, 
        \kappa(\mathbf{x}) 
         \, \mathrm{grad}[\beta(\mathbf{x})]
        \big\rrbracket
        \, \mathrm{d} \Gamma 
        - \int_{\Sigma} 
        \big\llbracket 
        \beta(\mathbf{x}) 
         \, 
        \kappa(\mathbf{x}) 
         \, \mathrm{grad}[\gamma(\mathbf{x})]
        \big\rrbracket
        \, \mathrm{d} \Gamma
        \notag \\
        &\qquad + \int_{\Omega \setminus \Sigma} 
        \beta(\mathbf{x}) \, 
        \mathrm{div}\big[ \kappa(\mathbf{x}) \, \mathrm{grad}[\gamma(\mathbf{x})] \big]
        \, \mathrm{d} \Omega  
    \end{align}
    }
\end{proposition}
\begin{proof} 
    Taking $\alpha(\mathbf{x}) = \gamma(\mathbf{x})$ and $\mathbf{a}(\mathbf{x}) = \kappa(\mathbf{x}) \, \mathrm{grad}[\beta(\mathbf{x})]$ in  Green's theorem \eqref{Eqn:Sensitivity_Greens_theorem}, we write:
    \begin{align}
        \label{Eqn:Sensitivity_App_Prop1_Step1}
        \int_{\Omega \setminus \Sigma} 
        \gamma(\mathbf{x}) \,  \mathrm{div}\big[\kappa(\mathbf{x}) \,  \mathrm{grad}[\beta(\mathbf{x})]\big] \, \mathrm{d} \Omega 
        &= \int_{\partial \Omega} 
        \gamma(\mathbf{x}) \, \kappa(\mathbf{x}) \, \mathrm{grad}[\beta(\mathbf{x})]
        \bullet \widehat{\mathbf{n}}(\mathbf{x}) 
        \, \mathrm{d} \Gamma 
        \notag \\
        &\qquad + \int_{\Sigma} 
        \big\llbracket 
        \gamma(\mathbf{x}) \, \kappa(\mathbf{x}) \, \mathrm{grad}[\beta(\mathbf{x})]
        \big\rrbracket
        \, \mathrm{d} \Gamma
        \notag \\
        &\qquad - \int_{\Omega \setminus \Sigma} 
        \mathrm{grad}[\gamma(\mathbf{x})] 
        \bullet \kappa(\mathbf{x}) \, \mathrm{grad}[\beta(\mathbf{x})]
        \, \mathrm{d} \Omega 
    \end{align}
    Using the commutative property of the dot product and invoking Green's theorem (with $\alpha(\mathbf{x}) = \beta(\mathbf{x})$ and $\mathbf{a}(\mathbf{x}) = \kappa(\mathbf{x}) \, \mathrm{grad}[\gamma(\mathbf{x})]$), the last integral (defined over $\Omega \setminus \Sigma$) in the above equation is rewritten as follows: 
    \begin{align}
        \label{Eqn:Sensitivity_App_Prop1_Step2}
        \int_{\Omega \setminus \Sigma} 
        \mathrm{grad}[\gamma(\mathbf{x})] 
        \bullet \kappa(\mathbf{x}) \, \mathrm{grad}[\beta(\mathbf{x})]
        \, \mathrm{d} \Omega 
        &= \int_{\Omega \setminus \Sigma} 
        \mathrm{grad}[\beta(\mathbf{x})] 
        \bullet \kappa(\mathbf{x}) \, \mathrm{grad}[\gamma(\mathbf{x})]
        \, \mathrm{d} \Omega \notag \\ 
        &= \int_{\partial \Omega} 
        \beta(\mathbf{x}) \, \kappa(\mathbf{x}) \, \mathrm{grad}[\gamma(\mathbf{x})]
        \bullet \widehat{\mathbf{n}}(\mathbf{x}) 
        \, \mathrm{d} \Gamma 
        \notag \\
        &\qquad + \int_{\Sigma} 
        \big\llbracket 
        \beta(\mathbf{x}) \, \kappa(\mathbf{x}) \, \mathrm{grad}[\gamma(\mathbf{x})]
        \big\rrbracket
        \, \mathrm{d} \Gamma
        \notag \\
        &\qquad - \int_{\Omega \setminus \Sigma} 
        \mathrm{grad}[\beta(\mathbf{x})] 
        \bullet \kappa(\mathbf{x}) \, \mathrm{grad}[\gamma(\mathbf{x})]
        \, \mathrm{d} \Omega 
    \end{align}
    By subtracting Eq.~ \eqref{Eqn:Sensitivity_App_Prop1_Step2} from Eq.~\eqref{Eqn:Sensitivity_App_Prop1_Step1}, we get the desired result:
    \begin{align}
        \label{Eqn:Sensitivity_App_Prop1_Step3}
        \int_{\Omega \setminus \Sigma} 
        \gamma(\mathbf{x}) \,  \mathrm{div}\big[\kappa(\mathbf{x}) \,  \mathrm{grad}[\beta(\mathbf{x})]\big] \, \mathrm{d} \Omega 
        &= \int_{\partial \Omega} 
        \gamma(\mathbf{x}) \, \kappa(\mathbf{x}) \, \mathrm{grad}[\beta(\mathbf{x})]
        \bullet \widehat{\mathbf{n}}(\mathbf{x}) 
        \, \mathrm{d} \Gamma 
        \notag \\
        &\qquad -
        \int_{\partial \Omega} 
        \beta(\mathbf{x}) \, \kappa(\mathbf{x}) \, \mathrm{grad}[\gamma(\mathbf{x})]
        \bullet \widehat{\mathbf{n}}(\mathbf{x}) 
        \, \mathrm{d} \Gamma 
        \notag \\
        &\qquad + \int_{\Sigma} 
        \big\llbracket 
        \gamma(\mathbf{x}) 
         \, 
        \kappa(\mathbf{x}) 
         \, \mathrm{grad}[\beta(\mathbf{x})]
        \big\rrbracket
        \, \mathrm{d} \Gamma \notag \\
        &\qquad - \int_{\Sigma} 
        \big\llbracket 
        \beta(\mathbf{x}) 
         \, 
        \kappa(\mathbf{x}) 
         \, \mathrm{grad}[\gamma(\mathbf{x})]
        \big\rrbracket
        \, \mathrm{d} \Gamma
        \notag \\
        &\qquad + \int_{\Omega \setminus \Sigma} 
        \beta(\mathbf{x}) \, 
        \mathrm{div}\big[ \kappa(\mathbf{x}) \, \mathrm{grad}[\gamma(\mathbf{x})] \big]
        \, \mathrm{d} \Omega  
    \end{align}
\end{proof}

\section{Derivation of the adjoint state problem}
\label{App:Sensitivity_Derivation_adjoint_problem}

To make the presentation concise in the main text, several intermediate steps were skipped in arriving at Eq.~\eqref{Eqn:Sensitivity_Expression_chi_collected_terms} from Eq.~\eqref{Eqn:Sensitivity_Expression_chi_original_terms}. Below we provide the missing details. 
We start with  Eq.~\eqref{Eqn:Sensitivity_Expression_chi_original_terms} and record that a superscript $\#$ is the Fr\'echet derivative with respect to $\chi$. Also, we note that the spatial derivatives (i.e., divergence and gradient operators) commute with the Fr\'echet derivative with respect to $\chi$. Since $\vartheta_{\mathrm{amb}}$, $f_0$, $d$, $h_T$, and $\kappa(\mathbf{x})$ do not depend on $\chi$, we rewrite Eq.~\eqref{Eqn:Sensitivity_Expression_chi_original_terms} as follows: 
\begin{align}
    \label{Eqn:Sensitivity_App_chi_Step1}
    D \Phi[\chi] 
    &= \int_{\Omega} \vartheta^{\#}(\mathbf{x};\chi) \, \mathrm{d} \Omega \notag \\
    & \quad \quad \quad +\frac{1}{f_0} \int_{\Omega \setminus \Sigma}
    \Big(\mu(\mathbf{x}) - \vartheta_{\mathrm{amb}}\Big) \,  \Big(d \, \mathrm{div}\big[\kappa(\mathbf{x}) \, \mathrm{grad}\big[\vartheta^{\#}(\mathbf{x};\chi)\big]\big] - h_{T} \, \vartheta^{\#}(\mathbf{x};\chi)\Big) 
    \, \mathrm{d} \Omega \notag \\
    &\quad \quad \quad -\frac{1}{f_0} \int_{\Sigma}  \Big\llangle\mu(\mathbf{x}) -  \vartheta_{\mathrm{amb}}\Big\rrangle \, \Big(\big\llbracket d \, 
    \kappa(\mathbf{x}) \,  \mathrm{grad}[\vartheta^{\#}(\mathbf{x};\chi)]\big\rrbracket + \chi \, \big\llangle  \mathrm{grad}[\vartheta^{\#}(\mathbf{x};\chi)]\bullet \widehat{\mathbf{t}}(\mathbf{x}) \big\rrangle 
    \notag\\
    &\hspace{3in}  + \big\llangle \mathrm{grad}[\vartheta(\mathbf{x};\chi)]\bullet \widehat{\mathbf{t}}(\mathbf{x}) \big\rrangle
    \Big) \, \mathrm{d} \Gamma \notag \\
    &\quad \quad \quad +\frac{1}{f_0} \int_{\Sigma} \Big\llangle d \, \kappa(\mathbf{x}) \mathrm{grad}[\mu(\mathbf{x})]\Big\rrangle
    \bullet 
    \big\llbracket\vartheta^{\#}(\mathbf{x};\chi)\big\rrbracket \, \mathrm{d} \Gamma \notag \\
    &\quad \quad \quad +\frac{1}{f_0} \int_{\partial \Omega} \Big(\mu(\mathbf{x}) -  \vartheta_{\mathrm{amb}}\Big) \, \Big(-d \, \kappa(\mathbf{x}) \,  \mathrm{grad}\big[\vartheta^{\#}(\mathbf{x};\chi)\big] \bullet \widehat{\mathbf{n}}(\mathbf{x})\Big) 
    \, \mathrm{d} \Gamma \notag \\
    &\quad \quad \quad -\frac{\chi}{f_0}
    \Big(\big\llangle\mu(\mathbf{x}) \big\rrangle - \vartheta_{\mathrm{amb}}\Big)
    \,  
    \big\llangle \vartheta^{\#}(\mathbf{x};\chi) 
    \big\rrangle \, \Big|_{s = 0 \;  (\mathrm{inlet})}
\end{align}
The central aim for the rest of the derivation is to isolate $\vartheta^{\#}(\mathbf{x};\kappa(\mathbf{x}))$ in each term of Eq.~\eqref{Eqn:Sensitivity_App_chi_Step1}.

For convenience, we denote the second term by $\mathcal{I}_2$: 
\begin{align}
    \label{Eqn:Sensitivity_App_chi_Step2}
    \mathcal{I}_{2}
    &= \frac{1}{f_0} \int_{\Omega \setminus \Sigma}
    \Big(\mu(\mathbf{x}) - \vartheta_{\mathrm{amb}}\Big) \,  \Big(d \, \mathrm{div}\big[\kappa(\mathbf{x})\mathrm{grad}[\vartheta^{\#}(\mathbf{x};\chi)]\big] 
    - h_{T} \, \vartheta^{\#}(\mathbf{x};\chi)\Big) \, \mathrm{d} \Omega \notag \\
    &= - \frac{1}{f_0} \int_{\Omega \setminus \Sigma}
    \vartheta^{\#}(\mathbf{x};\chi)
    \, h_T \, 
    \Big(\mu(\mathbf{x}) - \vartheta_{\mathrm{amb}}\Big) 
    \, \mathrm{d} \Omega 
    \notag \\
    &\qquad \qquad + \frac{1}{f_0} \int_{\Omega \setminus \Sigma}
    \Big(\mu(\mathbf{x}) - \vartheta_{\mathrm{amb}}\Big) \, 
    d \, \mathrm{div}\big[\kappa(\mathbf{x})\mathrm{grad}[\vartheta^{\#}(\mathbf{x};\chi)]\big] \, \mathrm{d} \Omega 
\end{align}
We now rewrite the integral $\mathcal{I}_2$ by moving the spatial derivatives on $\vartheta(\mathbf{x};\chi)$ to $\mu(\mathbf{x})$. By invoking Proposition \ref{Prop:Sensitivity_Weighted_theorem} on the second integral in Eq.~\eqref{Eqn:Sensitivity_App_chi_Step2}, we get the following: 
\begin{align}
    \label{Eqn:Sensitivity_App_Step3}
    \mathcal{I}_2 
    &= - \frac{1}{f_0} \int_{\Omega \setminus \Sigma}
    \vartheta^{\#}(\mathbf{x};\chi)
    \, h_T \, 
    \Big(\mu(\mathbf{x}) - \vartheta_{\mathrm{amb}}\Big) 
    \, \mathrm{d} \Omega \notag \\
    &\qquad +
    \frac{1}{f_0} \int_{\Omega \setminus \Sigma}
    \vartheta^{\#}(\mathbf{x};\chi) \, 
    \Big(
    d \, \mathrm{div}\big[\kappa(\mathbf{x}) \, 
    \mathrm{grad}[\mu(\mathbf{x})] \big] \Big) \, \mathrm{d} \Omega 
    \notag \\
    &\qquad + 
    \frac{1}{f_0} \int_{\partial \Omega}
    \Big(\mu(\mathbf{x}) - \vartheta_{\mathrm{amb}}\Big) \, d \, \kappa(\mathbf{x})\mathrm{grad}[\vartheta^{\#}(\mathbf{x};\chi)] \bullet \widehat{\mathbf{n}}(\mathbf{x}) \, \mathrm{d} \Gamma
    \notag \\
    &\qquad +
    \frac{1}{f_0} \int_{\Sigma}
    \Big\llbracket
    \left(\mu(\mathbf{x}) - \vartheta_{\mathrm{amb}} \right) 
    \bullet 
    d \, \kappa(\mathbf{x})\mathrm{grad}[\vartheta^{\#}(\mathbf{x};\chi)] 
    \Big\rrbracket \, \mathrm{d} \Gamma 
    \notag \\
    &\qquad - 
    \frac{1}{f_0} \int_{\partial \Omega}
    \vartheta^{\#}(\mathbf{x};\chi) \, 
    \Big(
    d \, \kappa(\mathbf{x}) \, 
    \mathrm{grad}[\mu(\mathbf{x})] \bullet
    \widehat{\mathbf{n}}(\mathbf{x}) \Big) \, \mathrm{d} \Gamma 
    \notag \\
    &\qquad 
    - \frac{1}{f_0} \int_{\Sigma}
    \Big\llbracket 
    \vartheta^{\#}(\mathbf{x};\chi) 
    \bullet 
    d \, \kappa(\mathbf{x}) \, 
    \mathrm{grad}[\mu(\mathbf{x})] 
    \Big \rrbracket \, \mathrm{d} \Gamma 
\end{align}

We now substitute the above expression into Eq.~\eqref{Eqn:Sensitivity_App_chi_Step1} and group the resulting terms into three categories for further simplification. We thus write:
\begin{align}
    \label{Eqn:Sensitivity_App_DPhi}
    D\Phi[\chi] 
    = \mathcal{J}_{1}
    + \mathcal{J}_{2}
    + \mathcal{J}_{3}
\end{align}
where $\mathcal{J}_1$ contains all the terms comprising integrals over $\Omega\setminus\Sigma$, 
$\mathcal{J}_{2}$ over $\partial \Omega$, and $\mathcal{J}_3$ consists of all the terms pertaining to the vasculature (i.e., terms containing integrals over $\Sigma$, and terms defined at the inlet or outlet of $\Sigma$).

The expression for $\mathcal{J}_{1}$ reads: 
\begin{align}
    \label{Eqn:Sensitivity_App_collecting_J1}
    \mathcal{J}_{1} 
    &= \int_{\Omega \setminus \Sigma} 
    \vartheta^{\#}(\mathbf{x};\chi) \, 
    \mathrm{d} \Omega 
    - \frac{1}{f_0} \int_{\Omega \setminus \Sigma}
    \vartheta^{\#}(\mathbf{x};\chi)
    \, h_T \, 
    \Big(\mu(\mathbf{x}) - \vartheta_{\mathrm{amb}}\Big) 
    \, \mathrm{d} \Omega \notag \\
    &\qquad \qquad \qquad +
    \frac{1}{f_0} \int_{\Omega \setminus \Sigma}
    \vartheta^{\#}(\mathbf{x};\chi) \, 
    \Big(
    d \, \mathrm{div}\big[\kappa(\mathbf{x}) \, 
    \mathrm{grad}[\mu(\mathbf{x})] \big] \Big) \, \mathrm{d} \Omega 
    \notag \\
    &= \frac{1}{f_0} \int_{\Omega \setminus \Sigma}
    \vartheta^{\#}(\mathbf{x};\chi) \, 
    \Big(
    d \, \mathrm{div}\big[\kappa(\mathbf{x}) \, 
    \mathrm{grad}[\mu(\mathbf{x})] \big] + f_0 - h_T \, 
    \big(\mu(\mathbf{x}) - \vartheta_{\mathrm{amb}}\big) \Big) \, \mathrm{d} \Omega 
\end{align}
The expression for $\mathcal{J}_{2}$ reads: 
\begin{align}
    \label{Eqn:Sensitivity_App_collecting_J2}
    \mathcal{J}_{2} 
    &= \frac{1}{f_0} \int_{\partial \Omega}
    \Big(\mu(\mathbf{x}) - \vartheta_{\mathrm{amb}}\Big) \, d \, \kappa(\mathbf{x})\mathrm{grad}[\vartheta^{\#}(\mathbf{x};\chi)] \bullet \widehat{\mathbf{n}}(\mathbf{x}) \, \mathrm{d} \Gamma
    \notag \\
    &\qquad \qquad \qquad -\frac{1}{f_0} \int_{\partial \Omega}
    \vartheta^{\#}(\mathbf{x};\chi) \, 
    \Big(
    d \, \kappa(\mathbf{x}) \, 
    \mathrm{grad}[\mu(\mathbf{x})] \bullet
    \widehat{\mathbf{n}}(\mathbf{x}) \Big) \, \mathrm{d} \Gamma 
    \notag \\ 
    &\qquad \qquad \qquad +
    \frac{1}{f_0} \int_{\partial \Omega} \Big(\mu(\mathbf{x}) -  \vartheta_{\mathrm{amb}}\Big) \, \Big(-d \, \kappa(\mathbf{x}) \mathrm{grad}\big[\vartheta^{\#}(\mathbf{x};\chi)\big] \bullet \widehat{\mathbf{n}}(\mathbf{x})\Big) \, \mathrm{d} \Gamma
    \notag \\
    &=\frac{1}{f_0} \int_{\partial \Omega}
    \vartheta^{\#}(\mathbf{x};\chi) \, 
    \Big(-d \, \kappa(\mathbf{x}) \, 
    \mathrm{grad}[\mu(\mathbf{x})] \bullet
    \widehat{\mathbf{n}}(\mathbf{x}) \Big) \, \mathrm{d} \Gamma 
\end{align}
The expression for $\mathcal{J}_{3}$ reads: 
\begin{align}
    \mathcal{J}_3 &=
    -\frac{1}{f_0} \int_{\Sigma}  \big\llangle\mu(\mathbf{x}) -  \vartheta_{\mathrm{amb}}\big\rrangle \, \big\llangle\mathrm{grad}[\vartheta(\mathbf{x};\chi)] 
    \bullet \widehat{\mathbf{t}}(\mathbf{x})
    \big\rrangle \, \mathrm{d} \Gamma 
    \notag\\
    &\qquad +\frac{1}{f_0} \int_{\Sigma}
    \left( 
    \Big\llbracket
    \left(\mu(\mathbf{x}) - \vartheta_{\mathrm{amb}} \right) 
    \bullet 
    d \, \kappa(\mathbf{x})\mathrm{grad}[\vartheta^{\#}(\mathbf{x};\chi)] 
    \Big\rrbracket
    -\Big\llangle\mu(\mathbf{x}) -  \vartheta_{\mathrm{amb}}\Big\rrangle \, 
    \Big\llbracket d \, \kappa(\mathbf{x}) \, \mathrm{grad}\big[\vartheta^{\#}(\mathbf{x};\chi)\big]\Big\rrbracket \right) \mathrm{d} \Gamma 
    \notag \\
     &\qquad 
    - \frac{1}{f_0} \int_{\Sigma}
    \left(\Big\llbracket 
    \vartheta^{\#}(\mathbf{x};\chi) 
    \bullet 
    d \, \kappa(\mathbf{x}) \, 
    \mathrm{grad}[\mu(\mathbf{x})] 
    \Big \rrbracket 
    - \Big\llangle d \, \kappa(\mathbf{x}) \mathrm{grad}[\mu(\mathbf{x})]\Big\rrangle
    \bullet 
    \big\llbracket\vartheta^{\#}(\mathbf{x};\chi)\big\rrbracket 
    \right) \, \mathrm{d} \Gamma \notag \\
    &\qquad -\frac{1}{f_0} \int_{\Sigma}  \big\llangle\mu(\mathbf{x}) -  \vartheta_{\mathrm{amb}}\big\rrangle \, 
    \big\llangle \chi \,   \mathrm{grad}\big[\vartheta^{\#}(\mathbf{x};\chi)\big]\bullet \widehat{\mathbf{t}}(\mathbf{x}) 
    \big\rrangle 
    \, \mathrm{d} \Gamma 
    \notag \\
    &\qquad -\frac{\chi}{f_0} \Big(\big\llangle\mu(\mathbf{x}) 
    \big \rrangle - \vartheta_{\mathrm{amb}}\Big) \,  \big\llangle \vartheta^{\#}(\mathbf{x};\chi) 
    \big \rrangle 
    \, \Big|_{s = 0 \;  (\mathrm{inlet})}
\end{align}
We simplify further by invoking the jump identities \eqref{Eqn:Sensitivities_avg_jump_identities} the second and third terms of the above equation: 
\begin{align}
    \mathcal{J}_3 &=
    -\frac{1}{f_0} \int_{\Sigma}  \big\llangle\mu(\mathbf{x}) -  \vartheta_{\mathrm{amb}}\big\rrangle \, 
    \big\llangle 
    \mathrm{grad}[\vartheta(\mathbf{x};\chi)] 
    \bullet \widehat{\mathbf{t}}(\mathbf{x})
    \big \rrangle \, \mathrm{d} \Gamma 
    \notag\\
    &\qquad +\frac{1}{f_0} \int_{\Sigma}
    \Big\llangle d \, \kappa(\mathbf{x}) \, \mathrm{grad}\big[\vartheta^{\#}(\mathbf{x};\chi)\big]
    \Big\rrangle
    \bullet 
    \Big\llbracket\mu(\mathbf{x}) -  \vartheta_{\mathrm{amb}}
    \Big\rrbracket 
    \, \mathrm{d} \Gamma 
    \notag \\
    &\qquad 
    - \frac{1}{f_0} \int_{\Sigma}
    \big\llangle\vartheta^{\#}(\mathbf{x};\chi)\big\rrangle
    \Big\llbracket d \, \kappa(\mathbf{x}) \mathrm{grad}[\mu(\mathbf{x})]\Big\rrbracket
    \, \mathrm{d} \Gamma \notag \\
    &\qquad -\frac{1}{f_0} \int_{\Sigma}  \big\llangle\mu(\mathbf{x}) -  \vartheta_{\mathrm{amb}}\big\rrangle \, 
    \big\llangle \chi \,  \mathrm{grad}\big[\vartheta^{\#}(\mathbf{x};\chi)\big]\bullet \widehat{\mathbf{t}}(\mathbf{x}) 
    \big\rrangle 
    \, \mathrm{d} \Gamma 
    \notag \\
    &\qquad -\frac{\chi}{f_0} \Big(\big\llangle\mu(\mathbf{x})\big\rrangle - \vartheta_{\mathrm{amb}}\Big) \, \big\llangle\vartheta^{\#}(\mathbf{x};\chi)\big\rrangle  
    \, \Big|_{s = 0 \;  (\mathrm{inlet})}
\end{align}

Noting that $\chi$ is independent of $\mathbf{x}$ and invoking Green's identity on the integral (in fact, it will be integration by parts, as the integral is in one spatial variable, $s$), the penultimate term can be rewritten as follows: 
\begin{align}
    &\frac{1}{f_0} \int_{\Sigma}  \big\llangle\mu(\mathbf{x}) -  \vartheta_{\mathrm{amb}}\big\rrangle \, 
    \big\llangle \chi \,  \mathrm{grad}\big[\vartheta^{\#}(\mathbf{x};\chi)\big]\bullet \widehat{\mathbf{t}}(\mathbf{x}) 
    \big\rrangle 
    \, \mathrm{d} \Gamma 
    \notag \\ 
    & \qquad \qquad = \frac{1}{f_0} \int_{0}^{1}  \big\llangle\mu(\mathbf{x}) -  \vartheta_{\mathrm{amb}}\big\rrangle \, 
    \big\llangle \chi \,  \frac{d}{ds}\vartheta^{\#}(\mathbf{x};\chi)
    \big\rrangle 
    \, \mathrm{d}s
    \notag \\ 
    & \qquad \qquad = \frac{\chi}{f_0} 
    \Big( \big\llangle\mu(\mathbf{x}) -  \vartheta_{\mathrm{amb}}\big\rrangle \, 
    \big\llangle  \vartheta^{\#}(\mathbf{x};\chi)
    \big\rrangle 
    \Big) \, \Big\vert_{s=0}^{s=1}
    - \frac{1}{f_0} \int_{0}^{1}  \chi \frac{d}{ds} \big\llangle \mu(\mathbf{x}) -  \vartheta_{\mathrm{amb}}\big\rrangle \, 
    \big\llangle \vartheta^{\#}(\mathbf{x};\chi)
    \big\rrangle 
    \, \mathrm{d}s
    \notag \\ 
    & \qquad \qquad = \frac{\chi}{f_0} 
    \Big( \big\llangle\mu(\mathbf{x}) \big \rrangle -  \vartheta_{\mathrm{amb}}\Big) \, 
    \big\llangle  \vartheta^{\#}(\mathbf{x};\chi)
    \big\rrangle 
    \, \Big\vert_{s=1 \; \mathrm{(outlet)}}
    - \frac{\chi}{f_0} 
    \Big( \big\llangle\mu(\mathbf{x}) 
    \big \rrangle - 
    \vartheta_{\mathrm{amb}} \, 
    \Big) 
    \big\llangle \vartheta^{\#}(\mathbf{x};\chi)
    \big\rrangle 
    \, \Big\vert_{s=0 \; \mathrm{(inlet)}}
    \notag \\
    &\qquad \qquad \qquad - \frac{1}{f_0} \int_{\Sigma}  
    \big\llangle \vartheta^{\#}(\mathbf{x};\chi)
    \big\rrangle 
    \big \llangle 
    \chi \, \mathrm{grad}[\mu(\mathbf{x})] 
    \bullet \widehat{\mathbf{t}}(\mathbf{x}) \big\rrangle 
    \, \mathrm{d}\Gamma
\end{align}
Using the above two equations, $\mathcal{J}_3$ will be written as:
\begin{align}
    \label{Eqn:Sensitivity_App_J3_terms_final}
    \mathcal{J}_3 &=
    -\frac{1}{f_0} \int_{\Sigma}  \big\llangle\mu(\mathbf{x}) -  \vartheta_{\mathrm{amb}}\big\rrangle \, 
    \big\llangle 
    \mathrm{grad}[\vartheta(\mathbf{x};\chi)] 
    \bullet \widehat{\mathbf{t}}(\mathbf{x})
    \big \rrangle \, \mathrm{d} \Gamma 
    \notag\\
    &\qquad +\frac{1}{f_0} \int_{\Sigma}
    \Big\llangle d \, \kappa(\mathbf{x}) \, \mathrm{grad}\big[\vartheta^{\#}(\mathbf{x};\chi)\big]
    \Big\rrangle
    \bullet 
    \Big\llbracket\mu(\mathbf{x}) -  \vartheta_{\mathrm{amb}}
    \Big\rrbracket 
    \, \mathrm{d} \Gamma 
    \notag \\
    &\qquad 
    - \frac{1}{f_0} \int_{\Sigma}
    \big\llangle\vartheta^{\#}(\mathbf{x};\chi)\big\rrangle
    \Big(\big\llbracket d \, \kappa(\mathbf{x}) \mathrm{grad}[\mu(\mathbf{x})]\big\rrbracket
    - \big\llangle \chi \,  \mathrm{grad}\big[\vartheta^{\#}(\mathbf{x};\chi)\big]\bullet \widehat{\mathbf{t}}(\mathbf{x}) 
    \big\rrangle 
    \Big) \, \mathrm{d} \Gamma 
    \notag \\
    &\qquad -\frac{\chi}{f_0}
    \big\llangle\vartheta^{\#}(\mathbf{x};\chi) \big\rrangle \, 
    \Big(\big\llangle\mu(\mathbf{x})\big\rrangle - \vartheta_{\mathrm{amb}}\Big) 
    \, \Big|_{s = 1 \;  (\mathrm{outlet})}
\end{align}

Finally, by substituting the terms for $\mathcal{J}_{1}$, $\mathcal{J}_{2}$ and $\mathcal{J}_{3}$---given by Eqs.~\eqref{Eqn:Sensitivity_App_collecting_J1}, \eqref{Eqn:Sensitivity_App_collecting_J2} and \eqref{Eqn:Sensitivity_App_J3_terms_final}---into Eq.~\eqref{Eqn:Sensitivity_App_DPhi}, we get
\begin{align}
    D\Phi[\chi] 
    &= -\frac{1}{f_0} \int_{\Sigma}  \big\llangle\mu(\mathbf{x}) -  \vartheta_{\mathrm{amb}}\big\rrangle \, 
    \big\llangle 
    \mathrm{grad}[\vartheta(\mathbf{x};\chi)] 
    \bullet \widehat{\mathbf{t}}(\mathbf{x})
    \big \rrangle \, \mathrm{d} \Gamma 
    \notag\\
    &\qquad + \frac{1}{f_0} \int_{\Omega \setminus \Sigma}
    \vartheta^{\#}(\mathbf{x};\chi) \, 
    \Big\{
    d \, \mathrm{div}\big[\kappa(\mathbf{x}) \, 
    \mathrm{grad}[\mu(\mathbf{x})] \big] + f_0 - h_T \, 
    \big(\mu(\mathbf{x}) - \vartheta_{\mathrm{amb}}\big) \Big\} \, \mathrm{d} \Omega 
    \notag \\
    &\qquad - 
    \frac{1}{f_0} \int_{\Sigma}
    \big\llangle\vartheta^{\#}(\mathbf{x};\chi)\big\rrangle
    \Big(\big\llbracket d \, \kappa(\mathbf{x}) \mathrm{grad}[\mu(\mathbf{x})]\big\rrbracket
    - \chi \, \big\llangle  \mathrm{grad}\big[\vartheta^{\#}(\mathbf{x};\chi)\big]\bullet \widehat{\mathbf{t}}(\mathbf{x}) 
    \big\rrangle 
    \Big) \, \mathrm{d} \Gamma 
    \notag \\
    &\qquad +\frac{1}{f_0} \int_{\Sigma}
    \Big\llangle d \, \kappa(\mathbf{x}) \, \mathrm{grad}\big[\vartheta^{\#}(\mathbf{x};\chi)\big]
    \Big\rrangle
    \bullet 
    \Big\{
    \Big\llbracket\mu(\mathbf{x}) \Big\rrbracket -  \vartheta_{\mathrm{amb}}
    \Big\} 
    \, \mathrm{d} \Gamma 
    \notag \\
    &\qquad +\frac{1}{f_0} \int_{\partial \Omega}
    \vartheta^{\#}(\mathbf{x};\chi) \, 
    \Big\{-d \, \kappa(\mathbf{x}) \, 
    \mathrm{grad}[\mu(\mathbf{x})] \bullet
    \widehat{\mathbf{n}}(\mathbf{x}) \Big\} \, \mathrm{d} \Gamma 
    \notag \\ 
    &\qquad -\frac{\chi}{f_0}
    \big\llangle\vartheta^{\#}(\mathbf{x};\chi) \big\rrangle \, 
    \Big\{\big\llangle\mu(\mathbf{x})\big\rrangle - \vartheta_{\mathrm{amb}}\Big\} 
    \, \Big|_{s = 1 \;  (\mathrm{outlet})}
\end{align}
which is same as Eq.~\eqref{Eqn:Sensitivity_Expression_chi_collected_terms}.

\section*{DATA AVAILABILITY}
The data that support the findings of this study are available from the corresponding author upon request.

\bibliographystyle{plainnat}
\bibliography{Master_References}
\end{document}